\documentclass[12pt,a4paper]{article}
\usepackage[a4paper, tmargin=3.4cm, bmargin=3.4cm, lmargin=2cm, rmargin=2cm, textheight=24cm, textwidth=16cm]{geometry}
\usepackage{setspace}
\usepackage{amsfonts}
\usepackage{epsfig}
\usepackage{latexsym}
\usepackage{amsthm}
\usepackage{mathtools,trimclip,lipsum}
\usepackage{hyperref}
\usepackage{xcolor}
\hypersetup{
	colorlinks,
	linkcolor={red!50!black},
	citecolor={blue!50!black},
	urlcolor={blue!80!black}
}
\usepackage{amsmath}
\usepackage{amssymb}

\usepackage[mathscr]{euscript}
 \let\mathscr\relax% just so we can load this and rsfs
\usepackage[scr]{rsfso}

\usepackage{array}
\usepackage{enumerate}
\usepackage{graphicx}
\usepackage{booktabs}
\usepackage{mathtools}
\usepackage{amsmath,blkarray,booktabs}

\newtheorem{theorem}{Theorem}[section]

\newtheorem{lemma}{Lemma}[section]

\newtheorem{re}{Remark}[section]
\newtheorem{nota}{Notation}[section]
\newtheorem{example}{Example}[section]

\providecommand{\keywords}[1]
{
	\small	
	\textbf{Keywords:} #1
}
\begin{document}

	\title{ Characteristic Polynomial of Power Graphs on Direct Product of Any Two Finite Cyclic Groups}
	
	\author{
		Komal Kumari and Pratima Panigrahi
		\\ \small Department of Mathematics, Indian Institute of Technology Kharagpur, India\\ \small e-mail: komalkumari1223w@gmail.com, pratima@maths.iitkgp.ac.in}
	
	%\date{}
	
	\maketitle
	\begin{abstract}
		The power graph $\mathscr{P}(G)$ of a group $G$ is defined as the simple graph with vertex set  $G$, and where two distinct vertices $x$ and $y$ are joined by an edge if and only if either $x= y^k$ or $y= x^k$, $k \in \mathbb{N}$.
  %This work aims to analyse the structure and spectrum of the  power graph of the group  $\mathbb{Z}_m \times \mathbb{Z}_{n}$, for any positive integers $m$ and $n$.
 Here we determine the characteristic polynomial of  $\mathscr{P}(\mathbb{Z}_m \times \mathbb{Z}_{n})$ for any positive integers $m$ and $n$.  Additionally, for some particular values of $m$ and $n$, we simplify the above characteristic polynomials and provide the full spectrum in a few cases.
	\end{abstract}
	
	
	\textbf{Mathematical Classification Code :}{05C75, 05C50, 05C25}
	
	\section{Introduction}
	All graphs considered in this article are simple and undirected. A graph $H$ is  typically denoted as $H=(V(H), E(H))$, where $V(H)$ is the vertex set and $E(H)$ is the edge set of $H$. For $u,$ $v \in V(H)$, by  the notation $u \sim v$, we mean $u$ and $v$  are adjacent. For $S_1,S_2 \subseteq V(H)$, $S_1 \cap S_2= \emptyset$, $S_1 \sim S_2$ means all the vertices in $S_1$ are adjacent with all the vertices in $S_2$. For $u \in V(H)$, $N(u)$ denotes the set of all neighbours of $u$. An \textit{r-regular graph} is the one in which every vertex is of degree $r$. The \textit{complete graph} $K_n$ on $n$ vertices, is an \textit{(n-1)-regular graph}.  The characteristic polynomial of a square matrix $N$ is defined by $det(xI-N)$, which we  represent as $\psi(N, x)$. For any simple graph $H$ with $n$ vertices, the adjacency matrix $A(H) =(a_{ij})$ is defined as an $n \times n$ matrix, where $a_{ij} =1$ if the $i^{th}$ vertex is adjacent to $j^{th}$ vertex, and zero otherwise. The \textit{spectrum} of $H$ is the multiset of all eigenvalues of $A(H)$ and is  denoted by  $\begin{pmatrix}
		\lambda_1 & \lambda_2 & \ldots & \lambda_r \\
		m_1 & m_2 &\ldots & m_r
	\end{pmatrix} $ where $ \lambda_1, \lambda_2, \ldots , \lambda_r$ are distinct eigenvalues of $A(H)$ with multiplicities $m_1,m_2, \ldots, m_r$ respectively. Kronecker product of two matrices $F$ and $N$, denoted by $F \bigotimes N$, is the block matrix  $[f_{ij} N]$ where $F= (f_{ij})$.\\ % More  intriguing  insights on the graphs can be found in \cite{west2001introduction}.\\
	
	A partition $V_1 \cup V_2 \cup \ldots \cup V_m$ of $V(H)$ of a graph $H$ is known as an \textit{equitable partition} if for all $i,j,1\leq i,j \leq m$, and  all $u,v \in V_i$, $|N(u) \cap V_j|=|N(v) \cap V_j| $ . For $v \in V_i$, we take  $q_{ij}=|N(v) \cap V_j|$, $1 \leq i,j  \leq m$, and $Q=(q_{ij})_{m \times m}$. The matrix $Q$ is called the \textit{quotient matrix} of the partition $V_1 \cup V_2 \cup \ldots \cup V_m $.  The $join$  $H_1 \vee H_2$ of graphs $H_1$ and $H_2$  is the graph obtained from the disjoint union of $H_1$ and $H_2$ by  connecting every vertex in $H_1$ to every vertex in $H_2$ by an edge. For an $n-$vertex  graph $H$ with vertex set $V(H) =\{1,2,...,n\}$ and a family of vertex disjoint graphs $H_1, H_2,...,H_n$,  the $H- generalized~join$ of  $H_1, H_2,..., H_n$  \cite{cardoso2013spectra}, denoted as $\tilde{G} = \bigvee_H \{ H_1, H_2, ...,H_n \}$, is a graph with vertex set $V(\tilde{G}) = \bigcup_{i=1}^{n}V(H_i)$  and edge set
	\begin{equation*}
		E(\tilde{G})=\left(\bigcup_{i=1}^{n}E(H_i)\right) \bigcup \left(\bigcup_{ij\in E(H)}\{uv: u\in V(H_i), v\in V(H_j)\}\right)
	\end{equation*} 
	If $n_i$ is the order of $H_i$, $1 \leq i \leq n$, and $v_{l_i} \in V(H_i)$  then $deg_{\tilde{G} (v_{l_i})} = deg_{H_i} (v_{l_i}) + \sum_{ij\in E(H)}n_j $. If $H=K_2$, then $H- generalized~join$ coincides with join of two graphs.\\
	\begin{theorem}[\cite{schwenk2006computing}] \label{theorem1}
		Let $H_i$ be $r_i$ regular graph, $1 \leq i \leq m$, and $H$ be an $m$-vertex graph. Then $V(H_1) \cup V(H_2) \cup \ldots \cup V(H_m)$ is an equitable partition  of $H[H_1 ,H_2 , \ldots,H_m]$ and the characteristic polynomial of  $H[H_1 ,H_2 , \ldots,H_m]$
		is  $\psi(H[H_1 ,H_2, \ldots,H_m],x) = \psi(Q,x) \Pi _{i=1} ^m \psi (H_i, x)/(x- r_i)$, where $Q$ is the quotient matrix. 
	\end{theorem} 
	Chakrabarty et al. \cite{chakrabarty2009undirected} defined ingeniously the \textit{power graph} $\mathscr{P}(G)$ of a group $G$, where the vertex set of $  \mathscr{P}(G)$ is $G$  and two distinct vertices are adjacent  when one is the integral power of the other. The direct product of two groups $G_1$ and $G_2$, is the group $G_1 \times G_2=\{(a,b):a \in G_1,b \in G_2\}$ where  the operation is component-wise. 
	% For additional intriguing results in group theory, we recommend \cite{dummit2004abstract}. \\ 
	Mehranian et al. \cite{mehranian2017spectra} studied the adjacency spectrum of power graphs of cyclic groups, elementary abelian groups of prime power order, the Mathieu group $M_{11}$ and the
	dihedral groups of order  twice a prime power. For distinct primes $p,q$ and $r$, Ghorbani et al. \cite{ghorbani2018characteristic} obtained the characteristic polynomial for power graphs of groups with order $p^3$ and $pqr$. Chattopadhyay et al. \cite{chattopadhyay2018spectral}  determined the characteristic polynomial of power graphs of  dihedral group and generalized quaternion group.  Bhuniya et al. \cite{bhuniya2017power} explored some relationship of  $\mathscr{P}(G_1 \times G_2)$ with $\mathscr{P}(G_1)$ and $\mathscr{P}(G_2)$. Fast forward to 2021, Jafari and Chattopadhyay \cite{jafari2022spectrum} contributed to the spectrum of proper power graphs of direct product of  finite element prime order groups (\textit{EPO -groups}).  For more interesting findings on the power graph, one may refer \cite{kumar2021recent}.\\

	%Inspired by the work of Jafari and Chattopadhyay \cite{jafari2022spectrum}, 
In this article we undertake an analysis of the spectrum of power graphs within the direct product of cyclic groups. In Section \ref{2}, we focus on the structural properties of $\mathscr{P}(\mathbb{Z}_{m} \times \mathbb{Z}_{n})$. In Section \ref{3}, we provide the  characteristic polynomial of $\mathscr{P}(\mathbb{Z}_m \times \mathbb{Z}_n)$, when $m|n$. In addition to this, we elaborate the characteristic polynomial of $\mathscr{P}(\mathbb{Z}_m \times \mathbb{Z}_n)$ for $m=p,$ $pq,$ and $p^2$.
	In Section \ref{4}, we find the full spectrum of $\mathscr{P}(\mathbb{Z}_p \times \mathbb{Z}_{pq})$  and  $\mathscr{P}(\mathbb{Z}_{p^2} \times \mathbb{Z}_{p^2})$.
    Finally, in Section \ref{5}, we obtain the characteristic polynomial of $\mathscr{P}(\mathbb{Z}_m \times \mathbb{Z}_n)$ when $m \nmid n$, and then provide explicitly the characteristic polynomial of  $\mathscr{P}(\mathbb{Z}_{p^2} \times \mathbb{Z}_n)$  for $p^2\nmid n$ but $p|n$. 
	%%%%%%%%%%%%%%%%%%%%%%%%%%%%%%%%%%%%%%%%%%%%%%%%%%%%%%%%%%%%%%%%%%%%%%%%%%%%%%%
	\section{Properties of $\mathscr{P}(G_1 \times G_2)$}\label{2}
	%Let $G$ be a finite group and $a_1\in G$ then the subgroup generated by $a_1$ is $\langle a \rangle$. Let $T_{a_1}$ be the set of generators of the subgroup  $\langle a_1 \rangle$. Now, take $a_2 \in G/ T_{a_1}$. Thus, the subgroup generated by $a_2$ is  $\langle a_2 \rangle$. And, let $T_{a_2}$ be the set of generators of the subgroup $\langle a_2 \rangle$. Similarly, we can do for all elements of $G$. Since $G$ is finite group, number of subsets in partition are finite. Thus, $a_i \in T_{a_i}$ for all $a_i \in G$ i.e 
	%\begin{equation}
	%	G= \bigcup_{a_i \in G} T_{a_i}
	%\end{equation}
	%To prove that distinct $T_{a_i}$ have null intersection, suppose $T_{a_i} \cap T_{a_j} \neq \emptyset.$ Let $ x \in T_{a_i} \cap T_{a_j}$ i.e $x$ is the generator of the groups $\langle a_i \rangle $ and $\langle a_j \rangle$. But, this is possible only when  $\langle a_i \rangle = \langle a_j \rangle$. 
In this section, we present some properties of $\mathscr{P}(G_1 \times G_2)$,  where $G_1$ and $G_2$ are any two finite groups.
	\begin{nota}{\rm
			We denote $J$, $I$, $M$ and $O$ as the  matrix comprising of all ones, the identity matrix, permutation matrix and the zero matrix of appropriate orders, respectively.	$E_{r \times s}$ denotes a  matrix of order $r \times s$ which has all one in last row and zero otherwise. $L_{r \times s}$ denotes a matrix of order $r \times s$ in which  $(r,s)^{th}$ element is $1$ and zero otherwise.
			For finite groups $G_1,G_2$ and  $(a,b) \in G_1 \times G_2$, $ \langle (a,b) \rangle$ denotes the subgroup of  $G_1 \times G_2$ generated by $ (a,b)$, and $|\langle(a,b)\rangle|$ denotes the order of $ \langle (a,b) \rangle$.  $T_{\langle (a,b) \rangle }$ denotes the set  of all generators of $ \langle (a,b) \rangle$.  We denote all the divisors of $n$ as $d_1,d_2,\ldots,d_k$, for some integer $k$. For any integer $m$, $\phi(m)$ denotes Euler's totient function. }
	\end{nota}
	We define a relation $\rho$ on $\mathscr{P}(G_1 \times G_2)$ as $(a_1,b_1)~ \rho~ (a_2,b_2)$ if and only if  $\langle { (a_1,b_1)}\rangle = \langle {(a_2,b_2)} \rangle$.
	Clearly,  $ \rho $  is an equivalence relation. The disjoint equivalence classes of $\rho$  are exactly $T_{\langle (a_1,b_1) \rangle}, T_{\langle (a_2,b_2) \rangle},$ $ \ldots, $ $T_{\langle (a_l,b_l) \rangle}$, for some positive integer $l$ and $(a_1,b_1), (a_2,b_2)\ldots, (a_l,b_l)\in V(\mathscr{P}(G_1 \times G_2)) $. Then we have the following partition of $V(\mathscr{P}(G_1 \times G_2))$,
	\begin{equation}
	V(\mathscr{P}(G_1 \times G_2))=	T_{\langle (a_1,b_1) \rangle} \cup T_{\langle (a_2,b_2) \rangle} \cup \ldots, T_{\langle (a_l,b_l) \rangle}
		\end{equation}
	Throughout the paper, we represent $A(\mathscr{P}(G))$ as a block matrix whose rows and columns are indexed by $T_{\langle (a_1,b_1) \rangle}, T_{\langle (a_2,b_2) \rangle},\ldots, T_{\langle (a_l,b_l) \rangle}$.
	\begin{lemma}[Corollary $2.9$, \cite{panda2018connectedness}]\label{lemma2.2}
		Let $S_1$, $S_2$ be distinct cyclic subgroups  of group  $G$ with generator sets $T_{S_1}$ and $T_{S_1}$ respectively. Then an  element of   ${T_{S_1}}$ is adjacent to an  element of   $T_{S_2}$ in $\mathscr{P}(G)$ if and only if all the elements of $T_{S_1}$ are adjacent to all the elements of $T_{S_2}$. 
	\end{lemma}

	\begin{lemma}\label{lema2.3}
		For $(a,b) \in G_1 \times G_2$, $T_{\langle {(a,b)} \rangle}$ induces a complete subgraph of $\mathscr{P}(G_1 \times G_2)$ of order ${\phi(|\langle (a,b)\rangle|)}$.
	\end{lemma}
	
	\begin{proof}
		As $T_{\langle {(a,b)} \rangle}$ is the generator set of the cyclic subgroup $\langle {(a,b)} \rangle$,  $|T_{\langle {(a,b)} \rangle}|= {\phi(|\langle (a,b)\rangle|)}$.  By the definition of  power graph, the result is straightforward.
	\end{proof}
	%\begin{proof}
	%	Let $a_1 \in T_{S_1}$, $ b_1 \in T_{S_2}$ with $a_1 \sim b_1$. Then  either $a_1 ^t = b_1$  or $a_1= b_1 ^t$. Without loss of generality, let   $a_1 ^t = b_1$. Let $ a$ and $b$ be arbitrary element of $ T_{S_1}$ and $T_{S_2}$ respectively. We need to prove that $a \sim b$. Since $a$ is a generator of $S_1$, we can write $a ^{l}= a_1$, for some integer $l$, then $(a^l)^t= b_1$. That means, $a$ is adjacent to $b_1$. Since $b_1$ is a generator of $S_2$, $b = b_1 ^k$ for some integer $k$. This implies $ a^{ltk}=b$. So, $a \sim b$.  Converse is obvious.    
	%\end{proof}

	\begin{lemma}\label{lemma2.3}
	The partition	$T_{\langle (a_1,b_1)\rangle} \cup T_{\langle (a_2,b_2) \rangle}\cup \ldots\cup T_{\langle (a_l,b_l)\rangle}$ (with respect to $\rho$) is an equitable partition of the vertex set of $\mathscr{P}(G_1 \times G_2)$ with $Q=(q_{ij})_{l \times l}$, where 
		\begin{equation}\label{eq2} 
			q_{ij}	= \begin{cases}
				\phi(|\langle(a_j,b_j)\rangle|) & \text{if~} i \neq j \text{~and~} T_{\langle (a_i,b_i)\rangle} \sim T_{\langle (a_j,b_j)\rangle} \\
				\phi(|\langle(a_i,b_i)\rangle|)-1 &  \text{if~}  i = j \\
				0 & \text{otherwise}\\
			\end{cases}
		\end{equation}
	\end{lemma}
	
	\begin{proof}
		%Let $T_{(a_1,b_1)},T_{(a_2,b_2)},\ldots,T_{(a_p,b_p)}$ be the all distinct equivalence classes of $G_1 \times G_2$ corresponding to $\rho$. So, $ T_{(a_1,b_1)} \cup T_{(a_2,b_2)}\cup  \ldots\cup  T_{(a_p,b_p)}$ is  a partition of vertex set of  $ \mathscr{P}(G_1 \times G_2)$.
		 From Lemma \ref{lemma2.2} and Lemma \ref{lema2.3}, the adjacency matrix of $\mathscr{P}(G_1 \times G_2)$ can be represented in terms of the following block matrix: 
		
		\[\scriptsize
		\begin{blockarray}{c ccccc}
			& T_{\langle (a_1,b_1)\rangle } & \cdots &T_{ \langle (a_j,b_j) \rangle } & \cdots & T_{\langle (a_l,b_l) \rangle}  \\
			%\cmidrule{2-6} 
			\begin{block}{c [ccccc]}
				T_{\langle (a_1,b_1) \rangle} & &  &\vdots  &  &  \\
				\vdots &  &  & \vdots &  &   \\
				T_{\langle (a_i,b_i) \rangle } &\cdots  &\cdots & B(a_i,a_j,b_i,b_j) & \cdots& \cdots  \\
				\vdots &  &  & \vdots &  &  \\
				T_{\langle (a_l,b_l) \rangle } &  &  & \vdots &  &   \\
			\end{block}
		\end{blockarray}
		\]

		where
		\begin{equation*} 
			B(a_i,a_j,b_i,b_j)= \begin{cases}
				J_{\phi(|\langle(a_i,b_i)\rangle|)\times \phi(|\langle(a_j,b_j)\rangle|)} & \text{if~} i \neq j \text{~and~} T_{\langle (a_i,b_i) \rangle } \sim T_{\langle (a_j,b_j)\rangle}  \\
				(J-I)_{\phi(|\langle(a_i,b_i)\rangle|)\times \phi(|\langle(a_i,b_i)\rangle|)} & \text{if~}  i = j  \\
				O_{\phi(|\langle(a_i,b_i)\rangle|)\times \phi(|\langle(a_j,b_j)\rangle|)} & \text{otherwise}\\
			\end{cases}
		\end{equation*} 
		We get that row sum of the matrix  $B(a_i,a_j,b_i,b_j)$ is $q_{ij}$ as given in equation (\ref{eq2}).
		Thus, the given  partition is  equitable.
	\end{proof}

%	\begin{lemma}
%	In $\mathscr{P}(\mathbb{Z}_{m} \times \mathbb{Z}_{n})$, a vertex of $T_{\langle (a_1,b_1)\rangle}$ is adjacent to  a vertex of $T_{\langle (a_2,b_2) \rangle}$ if and only if  there exists an integer $s$ such that either $(a_1,b_1) = \big(sa_2 (mod~ m),s b_2(mod~n)\big)$ or $(a_2,b_2) = \big(sa_1( mod~ m),s b_1(mod~n)\big)$.		
%\end{lemma}
%\begin{proof}
%	Suppose a vertex of $T_{\langle (a_1,b_1)\rangle}$ is adjacent with a vertex of $T_{\langle (a_2,b_2) \rangle}$. Then by Lemma \ref{lemma2.2} all the elements in $T_{\langle (a_1,b_1)\rangle}$ are adjacent to all the vertices in $T_{\langle (a_2,b_2) \rangle}$. We know that $(a_1,b_1) \in T_{\langle (a_1,b_1)\rangle}$ and $(a_2,b_2) \in T_{\langle (a_2,b_2) \rangle}$. Let   $(a_1,b_1) \sim (a_2,b_2)$  in $\mathscr{P}(\mathbb{Z}_{m} \times \mathbb{Z}_{n})$. Then there exists an integer $ s$ such that, without loss of generality, $(a_1,b_1)= (a_2,b_2)^s$. This implies $(a_1,b_1) \equiv (sa_2 (mod~ m),s b_2(mod~n)\big)$.\\
	%
%	Conversely, without loss of generality, let $a_1 \equiv sa_2 (mod~ m)$ and  $b_1\equiv sb_2(mod~n)$. This implies  $(a_1,b_1)= (a_2,b_2)^s$ in $\mathscr{P}(\mathbb{Z}_{m} \times \mathbb{Z}_{n})$. This implies,  $(a_1,b_1)$ and $(a_2,b_2)$ are adjacent in $\mathscr{P}(\mathbb{Z}_{m} \times \mathbb{Z}_{n})$. By Lemma \ref{lemma2.2}, each vertex of  $T_{\langle (a_1,b_1)\rangle}$ is adjacent to each vertex $T_{\langle (a_2,b_2) \rangle}$.
%\end{proof}

\begin{lemma}\label{lemma2.5}
	In $\mathscr{P}(\mathbb{Z}_{m} \times \mathbb{Z}_{n})$, a vertex of $T_{\langle (a_i,b_i)\rangle}$ is adjacent to  a vertex of $T_{\langle (a_i,b_i) \rangle}$ if and only if either $\langle (a_i,b_i)\rangle \subset \langle (a_j,b_j)\rangle$ or $\langle (a_j,b_j)\rangle \subset \langle (a_i,b_i)\rangle$ in $\mathbb{Z}_{m} \times \mathbb{Z}_{n}$.
\end{lemma}
\begin{proof}
	Suppose a vertex of $T_{\langle (a_i,b_i)\rangle}$ is adjacent with a vertex of $T_{\langle (a_j,b_j) \rangle}$. By Lemma \ref{lemma2.2}, all vertices of $T_{\langle (a_i,b_i)\rangle}$ are adjacent to all  vertex of $T_{\langle (a_j,b_j)\rangle}$.  Without loss of generality, by the definition of power graph there exists an integer $t$ such that $(a_i,b_i)^t= (a_j,b_j)$. This implies that $(a_j,b_j) \in \langle (a_i,b_i) \rangle$. Then we get $\langle (a_j,b_j)\rangle \subset \langle (a_i,b_i)\rangle$.\\
	
	Conversely, without loss of generality, let  $\langle (a_j,b_j)\rangle \subset \langle (a_i,b_i)\rangle$. Then $ (a_j,b_j) \in \langle (a_i,b_i)\rangle$. That means, there exists an integer $t$ such that $(a_i,b_i)^t = (a_j,b_j)$. This implies that $(a_i,b_i)$ is adjacent to $(a_j,b_j)$ in $\mathscr{P}(\mathbb{Z}_{m} \times \mathbb{Z}_{n})$.
\end{proof}

\begin{lemma}\label{lema2.5}
	 Let $H$ be a graph with $V(H)= \{1,2,\ldots,l\}$  and 	where $i \sim j$  if and only if either $\langle (a_i,b_i)\rangle \subset \langle (a_j,b_j)\rangle$ or $\langle (a_j,b_j)\rangle \subset \langle (a_i,b_i)\rangle$ in $\mathbb{Z}_{m} \times \mathbb{Z}_{n}$. Then
	\begin{equation*}
		\mathscr{P}(\mathbb{Z}_{m} \times \mathbb{Z}_{n}) = H[K_{\phi(|\langle (a_1,b_1)\rangle|)}, K_{\phi(|\langle (a_2,b_2)\rangle|)},\ldots, K_{\phi(|\langle (a_l,b_l)\rangle|)}].
	\end{equation*}

\end{lemma}
	\begin{proof}
	  The result follows from Lemma \ref{lemma2.5}. Lemma \ref{lema2.3} and Lemma \ref{lemma2.2}.
	\end{proof}

\begin{lemma}\cite{cvetkovic2009introduction}\label{lemma2.6}
	The spectrum of the complete graph $K_n$ is
	 \begin{equation*}
		\begin{pmatrix}
			n-1 & -1\\
			1& n-1
		\end{pmatrix}
	\end{equation*}
	
\end{lemma}
	%\begin{lemma}\label{lemma2.1}
	%	Let $x$, $y \in \mathbb{Z}_n$, $d_1= gcd(x,n)$ and $d_2 = gcd(y,n)$. Then $o(x)|o(y)$ if and only if $d_2|d_1$.
	%\end{lemma}
	%\begin{proof}
	%	Let $x = m_1d_1$, $y = m_2d_2$, $m_1$, $m_2$ $\in \mathbb{Z}$. Then $gcd \big(m_1, \frac{n}{d_1} \big)= 1$. Let $o(x)= t$. Then, $ (m_1 d_1)^t= nl$, for the some integer $l$. Since $\mathbb{Z}_n$ is an additive group,   $t.m_1d_1 = nl$. Then $ t= (\frac{n}{d_1})(\frac{l}{m_1})$. Since $gcd(m_1, \frac{n}{d_1})=1$,  $m_1|l$. Since $t$ is the smallest integer with  $ t= (\frac{n}{d_1})(\frac{l}{m_1})$ and $m_1|l$, we get $l= m_1$.  So, $o(x)= \frac{n}{d_1}$. Similarly, $o(y)= \frac{n}{d_2}$. Now, if $d_2|d_1$ then $o(x)|o(y)$. Converse is obvious. 	    
	%	\end{proof}
%%%%%%%%%%%%%%%%%%%%%%%%%%%%%%%%%%%%%%%%%%%%%%%%%%%%%%%%%%%%%%%%%%%%%%%%%%%
\section{Characteristic polynomial of $\mathscr{P}(\mathbb{Z}_{m} \times \mathbb{Z}_{n})$ with  $m|n$}\label{3}

In this section, since we are taking $m|n$, the order of cyclic subgroups of $\mathbb{Z}_m \times \mathbb{Z}_n$ are divisors of $n$. In the rest of the paper, we take $S=\{d_1,d_2,\ldots, d_k\}$ as the set of all divisors of $n$. 
%So from the definition of the equivalence relation $\rho$, considered in Section \ref{2}, we get the disjoint equivalence classes of $V(\mathscr{P}(\mathbb{Z}_m \times \mathbb{Z}_n))$ as $T_{\langle(a_i, \frac{n}{d_i}) \rangle}, i=1,2,\ldots,l $, where order of $a_i$ is a divisor of $d_i$.
\begin{theorem}\label{theorem3}
	Let $G = \mathbb{Z}_{m} \times \mathbb{Z}_n$ with $m|n$.  The characteristic polynomial of $A(\mathscr{P}(G))$ is 
	\begin{equation*}
		\psi(A(\mathscr{P}(G)),x) = (1+x)^{\alpha}~ \psi(Q,x)
	\end{equation*}
where $\alpha = \sum_{i=1} ^k n_i(\phi(d_i)-1)$, $n_i$ is the number of cyclic subgroups of $G$ of order $d_i$ and $Q =(q_{ij})_{l \times l}$ with
\begin{equation*} 
	q_{ij}	= \begin{cases}
		\phi(|\langle (a_j,b_j) \rangle|) & \text{if~} i \neq j \text{~and~} T_{\langle (a_i,b_i) \rangle } \sim T_{\langle (a_j,b_j) \rangle} \\
		\phi(|\langle (a_i,b_i) \rangle|)-1 &  \text{if~}  i = j \\
		0 & \text{otherwise}\\
	\end{cases}
\end{equation*}
\end{theorem}
\begin{proof}
	As $m|n$, all the elements of the group $G$ have possible orders $d_1,d_2,\ldots,d_k$. This implies that all the cyclic subgroups of $G$ are of orders $d_1,d_2,\ldots,d_k$. Let $n_1, n_2,\ldots,n_k$ be positive integers with $\sum_{i=1} ^{k} n_i=l$. Without loss of generality, suppose  $\langle(a_1, b_1)\rangle,\langle(a_2,b_2)\rangle, \ldots, \langle(a_{n_1},b_{n_1})\rangle$ are the cyclic subgroups of $G$ of order $d_1$, and for $1\leq t \leq k-1$, $\langle(a_{\sum_{i=1} ^{t} n_i+1} ,b_{\sum_{i=1} ^{t} n_i+1})\rangle,\langle(a_{\sum_{i=1} ^{t} n_i+2},b_{\sum_{i=1} ^{t} n_i+2})\rangle, \ldots, $ $ \langle(a_{\sum_{i=1} ^{t+1} n_i},b_{\sum_{i=1} ^{t+1} n_i})\rangle$  are the cyclic subgroups of $G$ of order $d_{t+1}$.\\
	
	We index the rows and columns of  $A(\mathscr{P}(G))$ in the order by the generator sets of the subgroups listed above. By Lemma \ref{lemma2.2} and Lemma \ref{lema2.3}, the block matrix of $\mathscr{P}(G)$ corresponding to $ T_{\langle(a_i,b_i)\rangle} \times T_{\langle(a_j,b_j)\rangle}$ is  $B\big(T_{\langle(a_i,b_i)\rangle},T_{\langle(a_j,b_j)\rangle}\big)$, where

\begin{equation*} 
B\big(T_{\langle(a_i,b_i)\rangle},T_{\langle(a_j,b_j)\rangle}\big)= \begin{cases}
		J_{\phi(|\langle(a_i,b_i)\rangle|)\times \phi(|\langle(a_j,b_j)\rangle|)} & \text{if~} i \neq j \text{~and~} T_{\langle(a_i,b_i)\rangle} \sim T_{\langle(a_j,b_j)\rangle }\\
		(J-I)_{\phi(|\langle(a_i,b_i)\rangle|)\times \phi(|\langle(a_i,b_i)\rangle|)} & \text{if~} i =j\\
		O_{\phi(|\langle(a_i,b_i)\rangle|)\times \phi(|\langle(a_j,b_j)\rangle|)} & \text{otherwise}\\
	\end{cases}
\end{equation*} 
 By  Lemma \ref{lemma2.3}, Lemma \ref{lema2.5}, Lemma \ref{lemma2.6} and Theorem \ref{theorem1}, we get the required result.
\end{proof}

We apply Theorem \ref{theorem3}, in the next few results, namely Theorems \ref{theorem3.1}, \ref{theorem3.2}, \ref{theorem3.3}, to study spectrum of $\mathscr{P}(\mathbb{Z}_m \times \mathbb{Z}_n)$ where $m = p$, $m =pq$ and $m=p^2$ respectively, for distinct primes $p$ and $q$.
 If $p \nmid n$ then $\mathbb{Z}_p \times \mathbb{Z}_n \cong \mathbb{Z}_{pn}$. The spectrum of $\mathbb{Z}_{n}$ is  discussed in \cite{mehranian2017spectra}. So in Theorem \ref{theorem3.1} below, we consider the case $p|n$. Before that, we present the following lemma.

\begin{lemma}\label{lemma3.1}
	Let prime $p$ divides $n$. Without loss of generality, suppose the first $r$ elements in $S$ are divisible by $p^2$, then the next $t$  elements in $S$ are divisible by $p$ but not by $p^2$, and the rest of the elements in $S$ are relatively prime to $p$. Then all possible cyclic subgroups of $\mathbb{Z}_p \times \mathbb{Z}_n$ are as given below:\\ 
	 $   {\langle(0,\frac{n}{d_i})\rangle},  {\langle(1,\frac{n}{d_i})\rangle},$ $ {\langle(2,\frac{n}{d_i})\rangle},$ $ \ldots, {\langle((p-1),\frac{n}{d_i})\rangle}$  for  $1 \leq i \leq r$; 
	${\langle(0,\frac{n}{d_i})\rangle},$ $ {\langle(1,\frac{n}{d_i})\rangle},$ $ {\langle(2,\frac{n}{d_i})\rangle}, \ldots,$ $ {\langle((p-1),\frac{n}{d_i})\rangle},$ $ {\langle(1,\frac{np}{d_i})\rangle}  $  for $r+1 \leq i \leq r+t$; and 
	${\langle(0,\frac{n}{d_i})\rangle} $  for $r+t+1 \leq i \leq k$.
\end{lemma}
\begin{proof}
	There may not be any $d_i$ which is divisible by $p^2$, and so $r$ may be zero. Since $p|n$, there exists at least one $d_i$ such that $p|d_i$ but $p^2 \nmid d_i$. So $t \geq 1$. And we have at least one $d_i$ which is relatively prime to $n$.
	We have following cases depending on values of $r$ and $t$.\\ 
	\textbf{Case:1}\label{case13.1}  If $r \geq 1$, then $d_1,d_2,\ldots,d_r$ are divisible by $p^2$. So, $gcd(p,d_i) =p$ and $gcd(p,\frac{d_i}{p})=p$ for $1\leq i \leq r$. The number of cyclic subgroups of  $G$ of order $d_i$   are $\frac{(\phi(1)\phi(d_i)+\phi(p)\phi(d_i))}{\phi(d_i)}=p$. These $p$ subgroups are $   {\langle(0,\frac{n}{d_i})\rangle},  {\langle(1,\frac{n}{d_i})\rangle},$ $ {\langle(2,\frac{n}{d_i})\rangle},$ $ \ldots, {\langle((p-1),\frac{n}{d_i})\rangle}$  for $1 \leq i \leq r$. \\ 
	\textbf{Case:2}\label{case23.1}  Here  $d_{r+1},d_{r+2},\ldots,d_{r+t}$ are divisible by $p$ but not by $p^2$. So, $gcd(p,d_i)=p$ and $gcd(p,\frac{d_i}{p})=1$ for $r+1 \leq i \leq r+t$. The number of cyclic subgroups of $G$ of  order $d_i$  are  $\frac{\phi(1)\phi(d_i)+\phi(p)\phi(d_i)+ \phi(p)\phi(\frac{d_i}{p})}{\phi(d_i)}= (p+1)$. These $(p+1)$ subgroups are  ${\langle(0,\frac{n}{d_i})\rangle},$ $ {\langle(1,\frac{n}{d_i})\rangle},$ $ {\langle(2,\frac{n}{d_i})\rangle}, \ldots,$ $ {\langle((p-1),\frac{n}{d_i})\rangle}, {\langle(1,\frac{np}{d_i})\rangle}  $  for  $r+1 \leq i \leq r+t$. \\  
	\textbf{Case:3}\label{case33.1} Here  $d_{r+t+1},d_{r+t+2},\ldots,d_{k}$ are relatively prime to $p$. So, $gcd(p,d_i)=1$ for $r+t+1 \leq i \leq k$ then the number of cyclic subgroup of  $G$ of  order $d_i$   is $\frac{(\phi(1)\phi(d_i))}{\phi(d_i)}=1$. This cyclic subgroup is ${\langle(0,\frac{n}{d_i})\rangle} $  for  $r+t+1 \leq i \leq k$.    
\end{proof}

 By Lemma \ref{lemma3.1}, there are total $pr+(p+1)t+(k-r-t)= l(say)$ cyclic subgroups of $\mathbb{Z}_p \times \mathbb{Z}_n$. 
We list all these in the order as follows : the cyclic subgroups of Case 1 (of Lemma \ref{lemma3.1}) are  $S_1 =  \langle(0,\frac{n}{d_1})\rangle$, $S_2 = \langle(1,\frac{n}{d_1})\rangle,\ldots,S_{pr}=\langle(p-1,\frac{n}{d_r})\rangle$, then the cyclic subgroup of Case 2 are $S_{pr+1}=\langle(0,\frac{n}{d_{r+1}})\rangle,S_{pr+2}=\langle(1,\frac{n}{d_{r+1}})\rangle, \ldots,S_{pr+p(t+1)}=\langle(1,\frac{n}{d_{r+t}})\rangle  $ and finally the cyclic subgroups of Case 3 are $S_{pr+p(t+1)+1}=\langle(0,\frac{n}{d_{r+t+1}})\rangle, \ldots, S_{l}=\langle(0,\frac{n}{d_{k}})\rangle$. 

\begin{theorem}\label{theorem3.1}
	Let $G= \mathbb{Z}_p \times \mathbb{Z}_n$ with $p |n$, $r$ and $t$ as given in Lemma \ref{lemma3.1}. Then,
	% the characteristic polynomial of $A(\mathscr{P}(G))$ is 
	\begin{equation*}
		\psi(A(\mathscr{P}(G)),x)= (1+x)^{\alpha} ~ \psi(Q,x)
	\end{equation*}
	where $\alpha = {p\sum_{i=1} ^{r} (\phi(d_i)-1)+ (p+1) \sum_{i=r+1} ^{r+t} (\phi(d_i)-1) + \sum_{i=r+t+1} ^{k} (\phi(d_i)-1) }$ and  $Q =(q_{ij})_{l \times l}$ with
	\begin{equation*} 
		q_{ij}= \begin{cases}
			\phi(|S_j|) & \text{if~} i \neq j \text{~and~} T_{S_i} \sim T_{S_j} \\
			\phi(|S_i|)-1 & \text{if~}  i = j\\
			0 & \text{otherwise}\\
		\end{cases}
	\end{equation*} 
\end{theorem}

\begin{proof}
    We have	$V(\mathscr{P}(\mathbb{Z}_p \times \mathbb{Z}_n)) = T_{S_1} \cup T_{S_2} \cup \ldots \cup T_{S_l} $.
	By Lemma \ref{lemma3.1}, number of cyclic subgroups of order $d_i$ is $ p$ for $1 \leq i \leq r$; $p+1$ for $r+1 \leq i \leq r+t$; $ 1$ for $r+t+1 \leq i \leq k$. 
	%Then we
	%index the rows and columns of $A(\mathscr{P}(G))$ in the order  of the above partition. By Lemma \ref{lemma2.2} and Lemma \ref{lema2.3}, the block matrix of $A(\mathscr{P}(G))$ corresponding to $ T_{S_i} \times T_{S_j}$ is  $M\big(T_{S_i},T_{S_j})$, where
	%\[
	%\begin{blockarray}{c ccccc}
	%	& T_{(0,\frac{n}{d_1})} & \cdots &T_{u,\frac{n}{d_i})} & \cdots & T_{(0,\frac{n}{d_k})}  \\
	%	\cmidrule{2-6} 
	%	\begin{block}{c [ccccc]}
		%		T_{(0,\frac{n}{d_1})} & &  &\vdots  &  &  \\
		%		\vdots &  &  & \vdots &  &   \\
		%		T_{(v,\frac{n}{d_j})} &\cdots  &\cdots & M(d_i,d_j,u,v) & \cdots& \cdots  \\
		%		\vdots &  &  & \vdots &  &  \\
		%		T_{(0,\frac{n}{d_k})} &  &  & \vdots &  &   \\
		%	\end{block}
	%\end{blockarray}
	%\]
	%\begin{equation*} 
	%	M\big(T_{S_i},T_{S_j}\big)= \begin{cases}
	%		J_{\phi(|S_i|)\times \phi(|S_j|)} & \text{if~} i \neq j \text{~and~} T_{S_i} \sim T_{S_j}\\
	%		(J-I)_{\phi(|S_i|)\times \phi(|S_i|)} & \text{if~} i =j\\
	%		O_{\phi(|S_i|)\times \phi(|S_j|)} & \text{otherwise}\\
	%	\end{cases}
	%\end{equation*} 
	%The above partition is equitable partition by Lemma \ref{lemma2.2}.  By Lemma \ref{lema2.5}  and
    Then by Lemma \ref{lema2.5},  and Theorem \ref{theorem3}, we get the required result.
    %the characteristic polynomial of $A(\mathscr{P}(G))$ is 
	%\begin{equation*}
	%	\psi(A(\mathscr{P}(G)),x)= (1+x)^{p\sum_{i=1} ^{r} (\phi(d_i)-1)+ (p+1) \sum_{i=r+1} ^{r+t} (\phi(d_i)-1) + \sum_{i=r+t+1} ^{k} (\phi(d_i)-1) }~\psi(Q,x)
	%\end{equation*}
%	where $Q=(q_{ij})_{ s\times s}$ with
	%	\[
	%	\begin{blockarray}{c ccccc} 
		%		& T_{(0,\frac{n}{d_1})} & \cdots &T_{(u,\frac{n}{d_i})} & \cdots & T_{(0,\frac{n}{d_k})}  \\
		%		\cmidrule{2-6} 
		%		\begin{block}{c [ccccc]}
			%			T_{(0,\frac{n}{d_1})} & &  &\vdots  &  &  \\
			%			\vdots &  &  & \vdots &  &   \\
			%			T_{v,\frac{n}{d_j})} &\cdots  &\cdots & Q(M(d_i,d_j,u,v)) & \cdots& \cdots  \\
			%			\vdots &  &  & \vdots &  &  \\
			%			T_{(0,\frac{n}{d_k})} &  &  & \vdots &  &   \\
			%		\end{block}
		%	\end{blockarray}
	%	\]
	%	where
%	\begin{equation*} 
%	q_{ij}= \begin{cases}
%		\phi(|S_j|) & \text{if~} i \neq j \text{~and~} T_{S_i} \sim T_{S_j} \\
%		\phi(|S_i|)-1 & \text{if~}  i = j\\
%		0 & \text{otherwise}\\
%	\end{cases}
 %  \end{equation*} 
\end{proof}

\begin{re}\rm{
		Taking  $n=p^2$ yields Corollary $3.2$ of  \cite{ghorbani2018characteristic} as a particular case of Theorem  \ref{theorem3.1}.}
\end{re}

If  $p \nmid n$ or $q \nmid n$ then $\mathbb{Z}_{pq} \times \mathbb{Z}_n$ is isomorphic to  $\mathbb{Z}_{q} \times \mathbb{Z}_{np}$ or $\mathbb{Z}_{p} \times \mathbb{Z}_{nq}$ respectively, and this case is discussed in Theorem \ref{theorem3.1}. So in the next theorem, we consider $pq|n$. 

\begin{lemma}\label{lemma3.2}
	For distinct primes $p$ and $q$, let $pq|n$. Without loss of generality, in $S$, suppose the first $r_1$ elements  are relatively prime to $pq$,  next $r_2$ elements  are divisible by $p$ but not by $p^2$ and $q$, next $r_3$ (may be zero) elements  are divisible by $p^2$ but not by $q$, next $r_4$ elements  are divisible by $q$ but not by $q^2$ and $p$,  next $r_5$ (may be zero) elements  are divisible by $q^2$ but not by $p$, next $r_6$ elements are divisible by $pq$ but not by $p^2q$ and $pq^2$, next $r_7$ (may be zero) elements are divisible by $p^2q$ but not by $pq^2$, next $r_8$ (may be zero) elements are divisible by $pq^2$ but not by $p^2q$, and the remaining elements (may be zero) are multiple of $p^2q^2$. Then all possible cyclic subgroup of $\mathbb{Z}_{pq} \times \mathbb{Z}_n$ are as given below:
	${\langle(0,\frac{n}{d_i})\rangle} $  for $1 \leq i \leq r_1;$ $ {\langle(0,\frac{n}{d_i})\rangle} ,$ ${\langle(q,\frac{n}{d_i})\rangle},$ $ {\langle(2q,\frac{n}{d_i})\rangle},$ $ \ldots, {\langle(q(p-1),\frac{n}{d_i})\rangle},{\langle(q,\frac{np}{d_i})\rangle}$   for $r_1+1\leq i \leq \sum_{j=1} ^2 r_j\leq \sum_{j=1} ^3 r_j;$
	$ {\langle(0,\frac{n}{d_i})\rangle} ,$ ${\langle(q,\frac{n}{d_i})\rangle},$ $ {\langle(2q,\frac{n}{d_i})\rangle},$ $ \ldots, {\langle(q(p-1),\frac{n}{d_i})\rangle}$   for $\sum_{j=1} ^2 r_j +1 \leq i\leq \sum_{j=1} ^3 r_j;$ $ {\langle(0,\frac{n}{d_i})\rangle} ,$ ${\langle(p,\frac{n}{d_i})\rangle},$ $ {\langle(2p,\frac{n}{d_i})\rangle},$ $ \ldots, {\langle(p(q-1),\frac{n}{d_i})\rangle},{\langle(p,\frac{nq}{d_i})\rangle}$   for  $\sum_{j=1} ^3 r_j +1 \leq i \leq \sum_{j=1} ^4 r_j \leq \sum_{j=1} ^5 r_j;$ $ {\langle(0,\frac{n}{d_i})\rangle} ,$ ${\langle(p,\frac{n}{d_i})\rangle},$ $ {\langle(2p,\frac{n}{d_i})\rangle},$ $ \ldots, {\langle(p(q-1),\frac{n}{d_i})\rangle}$   for $\sum_{j=1} ^4 r_j +1 \leq i \leq \sum_{j=1} ^5 r_j;$ $ {\langle(0,\frac{n}{d_i})\rangle} ,$ ${\langle(1,\frac{n}{d_i})\rangle},$ $ {\langle(2,\frac{n}{d_i})\rangle},$ $ \ldots, {\langle(pq-1,\frac{n}{d_i})\rangle},$ ${\langle(1,\frac{np}{d_i})\rangle}, $ ${\langle(1,\frac{n2p}{d_i})\rangle},\ldots,{\langle(1,\frac{np(q-1)}{d_i})\rangle}, $ ${\langle(q,\frac{np}{d_i})\rangle},$ ${\langle(1,\frac{nq}{d_i})\rangle},$ ${\langle(1,\frac{n2q}{d_i})\rangle},$ $\ldots,$ ${\langle(1,\frac{n(p-1)q}{d_i})\rangle},  {\langle(p,\frac{nq}{d_i})\rangle}, {\langle(1,\frac{npq}{d_i})\rangle}$   for   $\sum_{j=1} ^5 r_j +1 \leq i \leq \sum_{j=1} ^6 r_j \leq k;$  $ {\langle(0,\frac{n}{d_i})\rangle} ,$ ${\langle(1,\frac{n}{d_i})\rangle},$ $ {\langle(2,\frac{n}{d_i})\rangle}, \ldots, {\langle(pq-1,\frac{n}{d_i})\rangle},$ ${\langle(p,\frac{nq}{d_i})\rangle},$ ${\langle(1,\frac{nq}{d_i})\rangle},$ ${\langle(2,\frac{nq}{d_i})\rangle},\ldots,$ ${\langle(p-1,\frac{nq}{d_i})\rangle}$   for  $\sum_{j=1} ^6 r_j < i \leq \sum_{j=1} ^7 r_j \leq k;$  $ {\langle(0,\frac{n}{d_i})\rangle} ,$ ${\langle(1,\frac{n}{d_i})\rangle},$ $ {\langle(2,\frac{n}{d_i})\rangle},$ $ \ldots, {\langle(pq-1,\frac{n}{d_i})\rangle},$ ${\langle(1,\frac{np}{d_i})\rangle},{\langle(1,\frac{n2p}{d_i})\rangle},\ldots,$ ${\langle(1,\frac{np(q-1)}{d_i})\rangle},$ ${\langle(q,\frac{np}{d_i})\rangle}$   for $\sum_{j=1} ^7 r_j $ $< i \leq \sum_{j=1} ^8 r_j \leq k;$ $ {\langle(0,\frac{n}{d_i})\rangle} ,$ ${\langle(1,\frac{n}{d_i})\rangle},$ $ {\langle(2,\frac{n}{d_i})\rangle},$ $ \ldots, {\langle(pq-1,\frac{n}{d_i})\rangle}$   for  $\sum_{j=1} ^8 r_j < i \leq k$.
\end{lemma}
\begin{proof}
	As $pq|n$, the order of an element  of $ \mathbb{Z}_{pq} \times \mathbb{Z}_{n}$  is a divisor of $n$. Then we have  the following cases depending on the values of $r_1,r_2,r_3,r_4,r_5,r_6,r_7$ and $r_8$.\\
	%  the next $r_2 \geq 0$ number of elements in $S$ be divisible by $p$ but not by $q$, the next $r_3 \geq 0$ number of elements in $S$ be divisible by $q$ but not by $p$, and the rest of elements be divisible by $pq$.\\
	\textbf{Case:1} We have at least one $d_i$ which is relatively prime to $pq$. Let   $gcd(pq,d_i)=1$ for  $1 \leq i \leq r_1$. Then the number of cyclic subgroup of  $G$  of order $d_i$   is $\frac{(\phi(1)\phi(d_i))}{\phi(d_i)}=1$. This cyclic subgroup is ${\langle(0,\frac{n}{d_i})\rangle} $  for  $1 \leq i \leq r_1$.\\ 
	\textbf{Case:2} Since $p|n$, there exists at least one $d_i$ such that $p|d_i$ but  $pq \nmid d_i$. Let $gcd(pq,d_i) =p$ for $r_1+1\leq i \leq \sum_{j=1} ^3 r_j \leq k$. Then the following two subcases arise depending upon $p^2|d_i$ or $p^2 \nmid d_i$:\\
	\textbf{Subcase:2(a)} If $gcd(pq,\frac{d_i}{p})=1$ for $r_1+1\leq i \leq \sum_{j=1} ^2 r_j\leq \sum_{j=1} ^3 r_j$, then
	the number of cyclic subgroups of $G$ of order $d_i$   are $\frac{\phi(1)\phi(d_i)+\phi(p)\phi(d_i)+\phi(p)\phi(\frac{d_i}{p})}{\phi(d_i)}=(p+1)$. These $(p+1)$ cyclic subgroups are $ {\langle(0,\frac{n}{d_i})\rangle} ,$ ${\langle(q,\frac{n}{d_i})\rangle},$ $ {\langle(2q,\frac{n}{d_i})\rangle},$ $ \ldots, {\langle(q(p-1),\frac{n}{d_i})\rangle},{\langle(q,\frac{np}{d_i})\rangle}$   for $r_1+1\leq i \leq \sum_{j=1} ^2 r_j\leq \sum_{j=1} ^3 r_j$. \\ 
	\textbf{Subcase:2(b)} If $\sum_{j=1} ^2 r_j < \sum_{j=1} ^3 r_j$ then there exists at least one $d_i$ such that $p^2|d_i$. Let  $gcd(pq,\frac{d_i}{p})=p$ for $\sum_{j=1} ^2 r_j +1 \leq i\leq \sum_{j=1} ^3 r_j$. Then
	the number of cyclic subgroups of  $G$ of order $d_i$  are $\frac{\phi(1)\phi(d_i)+\phi(p)\phi(d_i)}{\phi(d_i)}=p$. These $p$ cyclic subgroups are $ {\langle(0,\frac{n}{d_i})\rangle} ,$ ${\langle(q,\frac{n}{d_i})\rangle},$ $ {\langle(2q,\frac{n}{d_i})\rangle},$ $ \ldots, {\langle(q(p-1),\frac{n}{d_i})\rangle}$   for   $\sum_{j=1} ^2 r_j +1 \leq i\leq \sum_{j=1} ^3 r_j$. \\ 
	\textbf{Case:3}  Since $q|n$, there exists at least one $d_i$ such that $q|d_i$ but not by $pq$. Let $gcd(pq,d_i) =q$ for  $\sum_{j=1} ^3 r_j +1 \leq i \leq \sum_{j=1} ^5 r_j$. Then the following two subcases arise depending upon $q^2|d_i$ or $q^2 \nmid d_i$:\\
	\textbf{Subcase:3(a)} If $gcd(pq,\frac{d_i}{q})=1$ for  $\sum_{j=1} ^3 r_j +1 \leq i \leq  \sum_{j=1} ^4 r_j \leq \sum_{j=1} ^5 r_j$, then
	the number of cyclic subgroups of $G$ of order $d_i$  are $\frac{\phi(1)\phi(d_i)+\phi(q)\phi(d_i)+\phi(q)\phi(\frac{d_i}{q})}{\phi(d_i)}=(q+1)$. These $(q+1)$ cyclic subgroups are $ {\langle(0,\frac{n}{d_i})\rangle} ,$ ${\langle(p,\frac{n}{d_i})\rangle},$ $ {\langle(2p,\frac{n}{d_i})\rangle},$ $ \ldots, {\langle(p(q-1),\frac{n}{d_i})\rangle},{\langle(p,\frac{nq}{d_i})\rangle}$   for   $\sum_{j=1} ^3 r_j +1 \leq i \leq \sum_{j=1} ^4 r_j \leq \sum_{j=1} ^5 r_j$. \\ 
	\textbf{Subcase:3(b)} If $\sum_{j=1} ^4 r_j < \sum_{j=1} ^5 r_j$ then there exists at least one $d_i$ such that $q^2|d_i$. Let $gcd(pq,\frac{d_i}{q})=q$ for $\sum_{j=1} ^4 r_j +1 \leq i \leq \sum_{j=1} ^5 r_j$. Then
	the number of cyclic subgroups of $G$ of order $d_i$   are $\frac{\phi(1)\phi(d_i)+\phi(q)\phi(d_i)}{\phi(d_i)}=q$. These $q$ cyclic subgroups are $ {\langle(0,\frac{n}{d_i})\rangle} ,$ ${\langle(p,\frac{n}{d_i})\rangle},$ $ {\langle(2p,\frac{n}{d_i})\rangle},$ $ \ldots, {\langle(p(q-1),\frac{n}{d_i})\rangle}$   for  $\sum_{j=1} ^4 r_j +1 \leq i \leq \sum_{j=1} ^5 r_j$. \\ 
	\textbf{Case:4} Since $pq|n$, there exists at least one $d_i$ such that $pq|d_i$. Let $gcd(pq,d_i) =pq$ for  $\sum_{j=1} ^5 r_j+1\leq i \leq k$. Then  there are four subcases arise:\\
	\textbf{Subcase:4(a)} Since $pq|d_i$, there exists at least one $d_i$   such that $gcd(pq,\frac{d_i}{pq})=1$ i.e $d_i = pq$. Let  $gcd(pq,\frac{d_i}{pq})=1$ for  $\sum_{j=1} ^5 r_j +1 \leq i \leq \sum_{j=1} ^6 r_j \leq k$. Then
	the number of cyclic subgroups of $G$ of order $d_i$   are $\frac{\phi(1)\phi(d_i)+\phi(p)\phi(d_i)+\phi(q)\phi(d_i)}{\phi(d_i)}+\frac{\phi(pq)\phi(d_i)+ \phi(pq)\phi(\frac{d_i}{p})+\phi(pq)\phi(\frac{d_i}{q})+\phi(pq)\phi(\frac{d_i}{pq})+\phi(p)\phi(\frac{d_i}{p})+\phi(q)\phi(\frac{d_i}{q})}{\phi(d_i)}=pq+p+q+1$. These $(pq+p+q+1)$ cyclic subgroups are $ {\langle(0,\frac{n}{d_i})\rangle} ,$ ${\langle(1,\frac{n}{d_i})\rangle},$ $ {\langle(2,\frac{n}{d_i})\rangle},$ $ \ldots, {\langle(pq-1,\frac{n}{d_i})\rangle},$ ${\langle(1,\frac{np}{d_i})\rangle}, $ ${\langle(1,\frac{n2p}{d_i})\rangle},\ldots,{\langle(1,\frac{np(q-1)}{d_i})\rangle}, $ ${\langle(q,\frac{np}{d_i})\rangle},$ ${\langle(1,\frac{nq}{d_i})\rangle},$ ${\langle(1,\frac{n2q}{d_i})\rangle},\ldots,$ ${\langle(1,\frac{n(p-1)q}{d_i})\rangle}, $ ${\langle(p,\frac{nq}{d_i})\rangle},$ ${\langle(1,\frac{npq}{d_i})\rangle}$   for   $\sum_{j=1} ^5 r_j +1 \leq i \leq \sum_{j=1} ^6 r_j \leq k$. \\
	If $\sum_{j=1} ^6 r_j < k$ then these three subcases may occur:\\
	\textbf{Subcase:4(b)} If   $gcd(pq,\frac{d_i}{pq})=p$ for $\sum_{j=1} ^6 r_j < i \leq \sum_{j=1} ^7 r_j \leq k$, then	the number of cyclic subgroups of  $G$ of order $d_i$   are $\frac{\phi(1)\phi(d_i)+\phi(p)\phi(d_i)+\phi(q)\phi(d_i)+\phi(pq)\phi(d_i)+\phi(q)\phi(\frac{d_i}{q})+\phi(pq)\phi(\frac{d_i}{q})}{\phi(d_i)}=pq+p$. These $(pq+p)$ cyclic subgroups are $ {\langle(0,\frac{n}{d_i})\rangle} ,$ ${\langle(1,\frac{n}{d_i})\rangle},$ $ {\langle(2,\frac{n}{d_i})\rangle},$ $ \ldots, {\langle(pq-1,\frac{n}{d_i})\rangle},{\langle(p,\frac{nq}{d_i})\rangle},{\langle(1,\frac{nq}{d_i})\rangle},$ ${\langle(2,\frac{nq}{d_i})\rangle},\ldots,$ ${\langle(p-1,\frac{nq}{d_i})\rangle}$   for  $\sum_{j=1} ^6 r_j < i \leq \sum_{j=1} ^7 r_j \leq k$.\\ 
	\textbf{Subcase:4(c)} If $gcd(pq,\frac{d_i}{pq})=q$ for $\sum_{j=1} ^7 r_j +1<  i \leq \sum_{j=1} ^8 r_j \leq k$, then
	the number of cyclic subgroups of $G$ of order $d_i$   are $\frac{\phi(1)\phi(d_i)+\phi(p)\phi(d_i)+\phi(q)\phi(d_i)+\phi(pq)\phi(d_i)+\phi(p)\phi(\frac{d_i}{p})+\phi(pq)\phi(\frac{d_i}{p})}{\phi(d_i)}=pq+q$. These $(pq+q)$ cyclic subgroups are $ {\langle(0,\frac{n}{d_i})\rangle} ,$ ${\langle(1,\frac{n}{d_i})\rangle},$ $ {\langle(2,\frac{n}{d_i})\rangle},$ $ \ldots, {\langle(pq-1,\frac{n}{d_i})\rangle},{\langle(1,\frac{np}{d_i})\rangle},{\langle(1,\frac{n2p}{d_i})\rangle},$ $\ldots,{\langle(1,\frac{np(q-1)}{d_i})\rangle},$ ${\langle(q,\frac{np}{d_i})\rangle}$   for  $\sum_{j=1} ^7 r_j < i \leq \sum_{j=1} ^8 r_j \leq k$. \\ 
	\textbf{Subcase:4(d)} If $gcd(pq,\frac{d_i}{pq})=pq$  for $\sum_{j=1} ^8 r_j < i \leq k$, then
	the number of cyclic subgroups of $G$ of order $d_i$  are $\frac{\phi(1)\phi(d_i)+\phi(p)\phi(d_i)+\phi(q)\phi(d_i)+\phi(pq)\phi(d_i)}{\phi(d_i)}=pq$. These $pq$ cyclic subgroups are $ {\langle(0,\frac{n}{d_i})\rangle} ,$ ${\langle(1,\frac{n}{d_i})\rangle},$ $ {\langle(2,\frac{n}{d_i})\rangle},$ $ \ldots, {\langle(pq-1,\frac{n}{d_i})\rangle}$   for  $\sum_{j=1} ^8 r_j < i \leq k$.
\end{proof}

 By Lemma \ref{lemma3.2}, there are total $r_1+(p+1)r_2+pr_3+(q+1)r_4+qr_5+(pq+p+q+1)r_6+(pq+p)r_7+(pq+q)r_8+pq(k-\sum_{j=1} ^{8}r_j)= l(say)$ cyclic subgroups of $\mathbb{Z}_{pq} \times \mathbb{Z}_n$. We list all these  in the order of their occurrence  in  the statement of Lemma \ref{lemma3.2}, that is $S_1 =  \langle(0,\frac{n}{d_1})\rangle$, $S_2 = \langle(0,\frac{n}{d_2})\rangle,\ldots,S_{l}=\langle(pq-1,\frac{n}{d_k})\rangle$.

\begin{theorem}\label{theorem3.2}
	Let $G= \mathbb{Z}_{pq} \times \mathbb{Z}_n$ with $pq|n$ and $r_1,r_2,r_3,r_4,r_5,r_6,r_7 ,r_8$ as given in Lemma \ref{lemma3.2}.
	% let first $r_1$ elements of $S$ be relatively prime to $pq$,  next $r_2$ elements of $S$ which are divisible by $p$ but not by $p^2$, next $r_3$ elements of $S$  which are divisible by $p^2$, next $r_4$ elements which are divisible by $q$ but not by $q^2$,  next $r_5$ elements which are divisible by $q^2$, next $r_6$ elements which are divisible by $pq$ but not by $p^2q$ or $pq^2$, next $r_7$ elements which are divisible by $p^2q$ but not by $pq^2$, next $r_8$ elements which are divisible by $pq^2$ but not by $p^2q$, and the remaining elements are multiple of $p^2q^2$. 
	Then,
	 %characteristic polynomial of  $A(\mathscr{P}(G))$ is 
	\begin{equation*}
		\psi(A(\mathscr{P}(G)),x)= (1+x)^{\alpha}~ \psi(Q,x)
	\end{equation*}
$ \alpha = \sum_{i=1} ^{t_1} (\phi(d_i)-1) +(p+1)\sum_{i=t_1+1} ^{t_2} (\phi(d_i)-1) +p\sum_{i=t_2+1} ^{t_3} (\phi(d_i)-1)+$ $ (q+1)\sum_{i=t_3+1} ^{t_4} (\phi(d_i)-1)$ $+ q\sum_{i=t_4+1} ^{t_5} (\phi(d_i)-1)$ $+(pq+p+q+1) \sum_{i=t_5+1} ^{t_6} (\phi(d_i)-1)+$ $(pq+p)\sum_{i=t_6+1} ^{t_7} (\phi(d_i)-1) +(pq+q)\sum_{i=t_7+1} ^{t_8} (\phi(d_i)-1) $ $+pq\sum_{i=t_8+1} ^k (\phi(d_i)-1) $ 
, and $t_1 =r_1, t_2 = \sum_{j=1} ^{2} r_j, t_3 = \sum_{j=1} ^{3} r_j, t_4 = \sum_{j=1} ^{4} r_j, t_5= \sum_{j=1} ^{5} r_j, t_6 = \sum_{j=1} ^{6} r_j, t_7 = \sum_{j=1} ^{7} r_j, t_8 = \sum_{j=1} ^{8} r_j$.
	 and  $Q=(q_{ij})_{l \times l}$ with
		\begin{equation*} 
		q_{ij}= \begin{cases}
			\phi(|S_j|) & \text{if~} i \neq j \text{~and~} T_{S_i} \sim T_{S_j} \\
			\phi(|S_i|)-1 & \text{if~}  i = j\\
			0 & \text{otherwise}\\
		\end{cases}
	\end{equation*} 
   
\end{theorem}
\begin{proof}
We have  $
V(\mathscr{P}(\mathbb{Z}_{pq} \times \mathbb{Z}_{n}))= T_{S_1} \cup T_{S_2} \cup \ldots \cup T_{S_l}$.
%	Then we
%index the rows and columns of $A(\mathscr{P}(G))$ in the order  of the above partition. By Lemma \ref{lemma2.2} and Lemma \ref{lema2.3}, the block matrix of $\mathscr{P}(G)$ corresponding to $ T_{S_i} \times T_{S_j}$ is  $M\big(T_{S_i},T_{S_j})$, where
%\[
%\begin{blockarray}{c ccccc}
%	& T_{(0,\frac{n}{d_1})} & \cdots &T_{u,\frac{n}{d_i})} & \cdots & T_{(0,\frac{n}{d_k})}  \\
%	\cmidrule{2-6} 
%	\begin{block}{c [ccccc]}
	%		T_{(0,\frac{n}{d_1})} & &  &\vdots  &  &  \\
	%		\vdots &  &  & \vdots &  &   \\
	%		T_{(v,\frac{n}{d_j})} &\cdots  &\cdots & M(d_i,d_j,u,v) & \cdots& \cdots  \\
	%		\vdots &  &  & \vdots &  &  \\
	%		T_{(0,\frac{n}{d_k})} &  &  & \vdots &  &   \\
	%	\end{block}
%\end{blockarray}
%\]
%\begin{equation*} 
%	M\big(T_{S_i},T_{S_j}\big)= \begin{cases}
%		J_{\phi(|S_i|)\times \phi(|S_j|)} & \text{if~} i \neq j \text{~and~} T_{S_i} \sim T_{S_j}\\
%		(J-I)_{\phi(|S_i|)\times \phi(|S_i|)} & \text{if~} i =j\\
%		O_{\phi(|S_i|)\times \phi(|S_j|)} & \text{otherwise}\\
%	\end{cases}
%\end{equation*}
By Lemma \ref{lemma3.2}, number of cyclic subgroups of order $d_i$ is $1$ for $1 \leq i \leq t_1$; $(p+1)$ for $t_1+1 \leq i \leq t_2$; $ p$ for $t_2 +1 \leq i \leq t_3$; $(q+1)$ for $t_3+1 \leq i \leq t_4$; $ q$ for $t_4+1 \leq i \leq t_5$; $ (pq+p+q+1)$ for $t_5+1 \leq i \leq t_6$; $ pq+p$ for $t_6+1 \leq i \leq t_7$; $ pq+q$ for $t_7+1 \leq i \leq t_8$; $ pq$ for $t_8+1 \leq i \leq k$.
%The above partition is equitable partition by Lemma \ref{lemma2.2}.  By Lemma \ref{lema2.5}  and 
Then by  Lemma \ref{lema2.5} and Theorem \ref{theorem3}, we get the required result.
%the characteristic polynomial of $A(\mathscr{P}(G))$ is 
%	\begin{equation*}
%		\psi(A(\mathscr{P}(G)),x)= (1+x)^{\alpha}
%		~ \psi(Q,x)
%	\end{equation*}

%	where $ \alpha = \sum_{i=1} ^{t_1} (\phi(|S_i|)-1) +(p+1)\sum_{i=t_1+1} ^{t_2} (\phi(|S_i|)-1) +p\sum_{i=t_2+1} ^{t_3} (\phi(|S_i|)-1)+$ $ (q+1)\sum_{i=t_3+1} ^{t_4} (\phi(|S_i|)-1)$ $+ q\sum_{i=t_4+1} ^{t_5} (\phi(|S_i|)-1)$ $+(pq+p+q+1) \sum_{i=t_5+1} ^{t_6} (\phi(|S_i|)-1)+$ $(pq+p)\sum_{i=t_6+1} ^{t_7} (\phi(|S_i|)-1) +(pq+q)\sum_{i=t_7+1} ^{t_8} (\phi(|S_i|)-1) $ $+pq\sum_{i=t_8+1} ^k (\phi(|S_i|)-1) $ 
%	, and $t_1 =r_1, t_2 = \sum_{j=1} ^{2} r_j, t_3 = \sum_{j=1} ^{3} r_j, t_4 = \sum_{j=1} ^{4} r_j, t_5= \sum_{j=1} ^{5} r_j, t_6 = \sum_{j=1} ^{6} r_j, t_7 = \sum_{j=1} ^{7} r_j, t_8 = \sum_{j=1} ^{8} r_j$.
	%and 
	%where $Q=(q_{ij})_{l \times l}$ with 
	%\begin{equation*} 
	%	q_{ij}= \begin{cases}
	%		\phi(d_j
	%		) & \text{if~} i \neq j \text{~and~} T_{(a_i,\frac{n}{d_i})} \sim T_{(a_j,\frac{n}{d_j})} \\
	%		\phi(d_i)-1 & \text{if~}  i = j\\
	%		0 & \text{otherwise}\\
	%	\end{cases}
	%\end{equation*}
\end{proof}

\begin{re}{\rm Theorem \ref{theorem3.2} can be extended to the case $\mathbb{Z}_m \times \mathbb{Z}_n$ when $m$ is product of distinct primes.
		Let $m = p_1p_2\ldots p_r$. Then the divisors of $m$ are in order  $1$, $p_i$, $p_i p_j (1 \leq i < j \leq r)$, $p_ip_jp_l (1 \leq i < j < l \leq r), \ldots,$ $ p_1p_2\ldots p_r$. We name these divisors as  $\xi_1,$ $\xi_2,$ $\ldots,$ $\xi_{2^r}.$ Analogous to Lemma \ref{lemma3.3}, we will have a lemma where there will be $2^r$ cases, one corresponding to each $\xi_i$, $1 \leq i \leq 2^r$. If $\xi_i = p_{i_1}p_{i_2}...p_{i_t}$ then the case corresponding to $\xi_i$ will have $2^t$ subcases. Number  of cyclic subgroups in a subcase corresponding to $\xi_i$ can be determined by the number of possible arrangement of divisors of $\xi_i$ and $d_j$(some divisor of $n$) such that there least common multiple is $d_j$. Finally we will have a theorem analogous to Theorem \ref{theorem3.2}.\\
		%Subcases of $\xi_i = p_{i_1}p_{i_2}...p_{i_t}$, number of subcases is $2^t$. 
		%Considering that some $d_j$ is divisible by $p_{i_{\beta}}$ for some $\beta \in \{1,2,\ldots,t\}$ and not divisible  by the rest of the factors of $\xi$, then divisible by $p_{i_{\beta}p_{i_{\gamma}}}$, $\beta, \gamma \in \{ 1,2,\ldots,t\}$ but not by other factors of $\xi_i$, and so on. Rows and columns of block representation of $A(\mathscr{P}(G))$ can be indexed in the following order: \\
		%Subgroups corresponding to subcases of $\xi_1,$ then subcases of $\xi_2$, and so on. This is one way to obtain $Q$ in a suitable form so that $\psi(Q,x)$ can be evaluated. 
		
	}	
\end{re}

If $p \nmid n$ then $Z_{p^2} \times Z_n \cong Z_{p^2n}=Z_{n_1}$ for some $n_1 \in \mathbb{N}$, and  this case is discussed in  \cite{mehranian2017spectra}. Thus we consider the case  $p^2|n$.

\begin{lemma}\label{lemma3.3}
    For prime $p$, let $p^2|n$. Without loss of generality, in $S$, suppose the first $r_1$ elements  are divisible by $p^3$,  the next $r_2$  elements are divisible by $p^2$ but not by $p^3$,  the next $r_3$ elements  are divisible by $p$ but not by $p^2$, and the rest of the elements are relatively prime to $p^2$. All possible cyclic subgroups of $\mathbb{Z}_{p^2} \times \mathbb{Z}_{n}$ are as given below:
    ${\langle(0,\frac{n}{d_i})\rangle},$ ${\langle(1,\frac{n}{d_i})\rangle},$ $ {\langle(2,\frac{n}{d_i})\rangle}, \ldots,$ $ {\langle(p^2-1,\frac{n}{d_i})\rangle}$  for  $1 \leq i \leq r_1$; 
    ${\langle(0,\frac{n}{d_i})\rangle},$ ${\langle(1,\frac{n}{d_i})\rangle},$ $ {\langle(2,\frac{n}{d_i})\rangle},$ $ \ldots,$ $ {\langle(p^2-1,\frac{n}{d_i})\rangle},$ $ {\langle(1,\frac{np}{d_i})\rangle},$ ${\langle(2,\frac{np}{d_i})\rangle},$ $\ldots $ $  {\langle((p-1),\frac{np}{d_i})\rangle},$ ${\langle(1,\frac{np^2}{d_i})\rangle} $  for  $r_1+1\leq i \leq r_1+r_2$; $ {\langle(0,\frac{n}{d_i})\rangle} ,$ ${\langle(p,\frac{n}{d_i})\rangle},$ $ {\langle(2p,\frac{n}{d_i})\rangle},$ $ \ldots, {\langle(p(p-1),\frac{n}{d_i})\rangle},{\langle(p,\frac{np}{d_i})\rangle}$   for $r_1+r_2+1 \leq i \leq r_1+r_2+r_3$; ${\langle(0,\frac{n}{d_i})\rangle} $  for $r_1+r_2+r_3+1 \leq i \leq k$. 
\end{lemma}
    \begin{proof}   As $p^2|n$, the order of all the elements of the group $G$ are the  divisors of $n$. There may not be any $d_i$ which is divisible by $p^3$, and so   $r_1$ may be zero. Since $p^2|n$, there exists at least one $d_i$ such that $p^2|d_i$ i.e. $r_2 \geq 1$. Since $p|n$, there exists at least one $d_i$ such that $p|d_i$ i.e. $r_3 \geq 1$.  And we have at least one $d_i$ which is relatively prime to $n$. We have following cases depending on values of $r_1,$ $r_2$ and $r_3$.  \\
	\textbf{Case:1} If $r_1 \geq 1$, then $d_1,d_2,\ldots,d_{r_1}$ are divisible by $p^3$. So, $gcd(p^2,d_i)=p^2$ and $gcd(p^2,\frac{d_i}{p})=p^2$ for  $1 \leq i \leq r_1$. The number of cyclic subgroups of  $G$ of order $d_i$  are  $\frac{\phi(1)\phi(d_i)+\phi(p)\phi(d_i)+ \phi(p^2)\phi(d_i)}{\phi(d_i)}= p^2$. These $p^2$ subgroups are  ${\langle(0,\frac{n}{d_i})\rangle},$ ${\langle(1,\frac{n}{d_i})\rangle},$ $ {\langle(2,\frac{n}{d_i})\rangle}, \ldots,$ $ {\langle(p^2-1,\frac{n}{d_i})\rangle}  $  for  $1 \leq i \leq r_1$. \\ 
	\textbf{Case:2} Here $d_{r_1+1},d_{r_1+2},\ldots,d_{r_1+r_2}$ are divisible by $p^2$ but not by $p^3$. So, $gcd(p^2,d_i)=p^2$ and $gcd(p^2,\frac{d_i}{p})=p$ for $r_1+1\leq i \leq r_1+r_2$. The number of cyclic subgroups of  $G$ of  order $d_i$ are  $\frac{\phi(1)\phi(d_i)+\phi(p)\phi(d_i)+ \phi(p^2)\phi(d_i)}{\phi(d_i)}$ $+$ $ \frac{\phi(p^2)\phi(\frac{d_i}{p})+\phi(p^2)\phi(\frac{d_i}{p^2})}{\phi(d_i)}= (p^2+p)$. These $(p^2+p)$ subgroups are  ${\langle(0,\frac{n}{d_i})\rangle},$ ${\langle(1,\frac{n}{d_i})\rangle},$ $ {\langle(2,\frac{n}{d_i})\rangle},$ $ \ldots,$ $ {\langle(p^2-1,\frac{n}{d_i})\rangle},$ $ {\langle(1,\frac{np}{d_i})\rangle},$ ${\langle(2,\frac{np}{d_i})\rangle},$ $\ldots $ $  {\langle((p-1),\frac{np}{d_i})\rangle},$ ${\langle(1,\frac{np^2}{d_i})\rangle} $  for $r_1+1\leq i \leq r_1+r_2$. \\ 
	\textbf{Case:3}  Here $d_{r_1+r_2+1},d_{r_1+r_2+2},\ldots,d_{r_1+r_2+r_3}$ are divisible by $p$ but not by $p^2$. So, $gcd(p^2,d_i) =p$ for  $r_1+r_2+1 \leq i \leq r_1+r_2+r_3$. The number of cyclic subgroups of  $G$ of order $d_i$  are $\frac{\phi(1)\phi(d_i)+\phi(p)\phi(d_i)+\phi(p)\phi(\frac{d_i}{p})}{\phi(d_i)}=(p+1)$. These $(p+1)$ cyclic subgroups are $ {\langle(0,\frac{n}{d_i})\rangle} ,$ ${\langle(p,\frac{n}{d_i})\rangle},$ $ {\langle(2p,\frac{n}{d_i})\rangle},$ $ \ldots, {\langle(p(p-1),\frac{n}{d_i})\rangle},{\langle(p,\frac{np}{d_i})\rangle}$   for  $r_1+r_2+1 \leq i \leq r_1+r_2+r_3$. \\
	\textbf{Case:4}   Here $d_{r_1+r_2+r_3+1},d_{r_1+r_2+r_3+2},\ldots,d_{k}$ are relatively prime to $p$. So, $gcd(p^2,d_i)=1$ for $r_1+r_2+r_3+1 \leq i \leq k$. The number of cyclic subgroup of $G$ of order $d_i$   is $\frac{(\phi(1)\phi(d_i))}{\phi(d_i)}=1$. This cyclic subgroup is ${\langle(0,\frac{n}{d_i})\rangle} $  for   $r_1+r_2+r_3+1 \leq i \leq k$.
\end{proof} 

 By Lemma \ref{lemma3.3}, there are total $p^2r_1+(p^2+p)r_2+(p+1)r_3+(k-r_1-r_2-r_3)= l(say)$ cyclic subgroups of $\mathbb{Z}_p \times \mathbb{Z}_n$. 
 We list all these  in the order of their occurrence  in  the statement of Lemma \ref{lemma3.3}, that is $S_1 =  \langle(0,\frac{n}{d_1})\rangle$, $S_2 = \langle(1,\frac{n}{d_1})\rangle,\ldots,S_{l}=\langle(0,\frac{n}{d_k})\rangle$.

\begin{theorem}\label{theorem3.3}
	Let $G= \mathbb{Z}_{p^2} \times \mathbb{Z}_n$ with $p^2|n$, $r_1,r_2$ and $r_3$ as given in Lemma \ref{lemma3.3}. 
	% Let the first $r$ elements in $S$ be divisible by $p^3$,  the next $r_2$  elements in $S$ be divisible by $p^2$ but not by $p^3$,  the next $r_3$ elements in $S$ be divisible by $p$ but not by $p^2$, and the rest of the elements in $S$ be relatively prime to $p^2$. 
	Then,
	% the  characteristic polynomial of  $A(\mathscr{P}(G))$ is 
	\begin{equation*}\scriptsize
		\psi(A(G)),x) = (1+x)^{\alpha}~\psi(Q,x)
	\end{equation*}
	where $\alpha= p\sum_{i=1} ^{r_1} (\phi(d_i)-1)+ (p+1) \sum_{i=r_1+1} ^{r_1+r_2} (\phi(d_i)-1) $ $+p^2  $ $\sum_{i=r_1+r_2+1} ^{r_1+r_2+r_3} (\phi(d_i)-1)  + $ $(p^2+p) \sum_{i=r_1+r_2+r_3+1} ^{k}$ $ (\phi(d_i)-1)$, $Q =(q_{ij})_{l \times l}$ with
	\begin{equation*} 
	q_{ij}= \begin{cases}
		\phi(|S_j|) & \text{if~} i \neq j \text{~and~} T_{S_i} \sim T_{S_j} \\
		\phi(|S_i|)-1 & \text{if~}  i = j\\
		0 & \text{otherwise}\\
	\end{cases}
   \end{equation*}
\end{theorem}
\begin{proof}
We have $
	V(\mathscr{P}(\mathbb{Z}_{p^2} \times \mathbb{Z}_{n}))= T_{S_1} \cup T_{S_2} \cup \ldots \cup T_{S_l}$.
%	Then we index the rows and columns of $A(\mathscr{P}(G))$ in the order  of the above partition. By Lemma \ref{lemma2.2} and Lemma \ref{lema2.3}, the block matrix of $\mathscr{P}(G)$ corresponding to $ T_{S_i} \times T_{S_j}$ is  $M\big(T_{S_i},T_{S_j})$, where
%\[
%\begin{blockarray}{c ccccc}
%	& T_{(0,\frac{n}{d_1})} & \cdots &T_{u,\frac{n}{d_i})} & \cdots & T_{(0,\frac{n}{d_k})}  \\
%	\cmidrule{2-6} 
%	\begin{block}{c [ccccc]}
	%		T_{(0,\frac{n}{d_1})} & &  &\vdots  &  &  \\
	%		\vdots &  &  & \vdots &  &   \\
	%		T_{(v,\frac{n}{d_j})} &\cdots  &\cdots & M(d_i,d_j,u,v) & \cdots& \cdots  \\
	%		\vdots &  &  & \vdots &  &  \\
	%		T_{(0,\frac{n}{d_k})} &  &  & \vdots &  &   \\
	%	\end{block}
%\end{blockarray}
%\]
%\begin{equation*} 
%	M\big(T_{S_i},T_{S_j}\big)= \begin{cases}
%		J_{\phi(|S_i|)\times \phi(|S_j|)} & \text{if~} i \neq j \text{~and~} T_{S_i} \sim T_{S_j}\\
%		(J-I)_{\phi(|S_i|)\times \phi(|S_i|)} & \text{if~} i =j\\
%		O_{\phi(|S_i|)\times \phi(|S_j|)} & \text{otherwise}\\
%	\end{cases}
%\end{equation*} 
%The above partition is equitable partition by Lemma \ref{lemma2.2}.  
%By Lemma \ref{lema2.5}  and
By Lemma \ref{lemma3.3}, number of cyclic subgroups of order $d_i$ is $p$ for $1 \leq i \leq r_1$; $(p+1)$ for $r_1+1 \leq i \leq r_1+r_2$; $ p^2$ for $r_1+r_2+1 \leq i \leq r_1+r_2+r_3$; $(p^2+p)$ for $r_1+r_2+r_3+1 \leq i \leq k$. By Lemma \ref{lema2.5},  and Theorem \ref{theorem3}, we get the required result.
 %the characteristic polynomial of $A(\mathscr{P}(G))$ is 
%	\begin{equation*}
%		\psi(A(\mathscr{P}(G))),x)= (1+x)^{\alpha}~ \psi(Q,x)
%	\end{equation*}
%	where $\alpha= p\sum_{i=1} ^{r_1} (\phi(d_i)-1)+$ $ (p+1) \sum_{i=r_1+1} ^{r_1+r_2} (\phi(d_i)-1) $ $+p^2  \sum_{i=r_1+r_2+1} ^{r_1+r_2+r_3}$ $ (\phi(d_i)-1) $ $ + (p^2+p) \sum_{i=r_1+r_2+r_3+1} ^{k} $ $(\phi(d_i)-1)$. 
	%and    where $Q=(q_{ij})_{t \times t}$ with
	%\begin{equation*} 
	%	q_{ij}= \begin{cases}
	%		\phi(d_j) & \text{if~} i \neq j \text{~and~} T_{(a_i,\frac{n}{d_i})} \sim T_{(a_j,\frac{n}{d_j})} \\
	%		\phi(d_i)-1 & \text{if~}  i = j\\
	%		0 & \text{otherwise}\\
	%	\end{cases}
	%\end{equation*}
\end{proof}

\section{Full Spectrum of Some $\mathscr{P}(\mathbb{Z}_m \times \mathbb{Z}_n )$}\label{4}

\begin{theorem}\label{theorem2.2} 
	Let $G = \mathbb{Z}_p \times \mathbb{Z}_{pq}$, where $p$, $q$ are  distinct prime numbers. Then the spectrum of $\mathscr{P}(G)$ consists of: 
	%	\begin{equation*}
		%		\begin{pmatrix}
			%		-1&\phi(pq)+\phi(p)-1&x_1&x_2&x_3&x_4\\
			%	    	p^2q-p-4& p& 1&1&1&1
			%		\end{pmatrix}
		%	\end{equation*}
	\begin{enumerate}
		\item $-1$ with multiplicity $p^2q-p-4$.
		\item $(p-1)q-1$ with multiplicity $p$.
		\item And the roots of the polynomial \\
		$-x^4+(pq-4)x^3+ (p^2q^2-p^2q-pq^2+3pq+p^2+q-7)x^2+(-p^3(q-1)^2+3p^2q^2-5p^2q-pq^2+4p^2-q^2+pq+p+5q-8)x -p^4(q-1)^2 +3p^2q^2-6p^2q-pq^2+pq+4p^2-q^2+4q-4$.
		 
	%	 \begin{equation*}
	%		\begin{vmatrix}
	%			p+c-ac& p+(1-a)+(1-d)\\
	%			b+c-bc& -p-b(1-d)
	%		\end{vmatrix}
	%	\end{equation*}
	\end{enumerate}
%	where $a=\frac{(\phi(pq)-1-x)}{\phi(pq)}$, $b= \frac{ (\phi(p)-1-x)}{\phi(p)}$, $c=\frac{(\phi(q)-1-x)}{\phi(q)}$, and $d=-x$
\end{theorem}
\begin{proof}
	For $G = \mathbb{Z}_p \times \mathbb{Z}_{pq}$, an element of $G$ is of the order $pq$, $q$, $p$ or $1$. We can easily deduce that $G$ has exactly $ \frac{\phi(pq) (\phi(p)+2)}{\phi(pq)}= (p+1)$ cyclic subgroups of order $pq$,  $ \frac{\phi(p) (\phi(p)+2)}{\phi(p)}= (p+1)$ cyclic subgroups of order $p$, one cyclic subgroup each of order $q$ and  $1$. 
	Then by Case 2 of the proof of  Lemma \ref{lemma3.2}  the cyclic  subgroups of order $pq$ are $\langle(0,1)\rangle, \langle(1,1) \rangle,\langle(2,1) \rangle, \ldots , \langle (p-1,1) \rangle,  \langle(1,p) \rangle$ and the cyclic subgroups of order $p$ are $\langle(0,q)\rangle, \langle(1,q) \rangle,$ $ ,\langle(2,q) \rangle,$ $\ldots , \langle (p-1,q) \rangle,  \langle(1,0) \rangle$. And by Case 3 of the proof of Lemma \ref{lemma3.2}  the cyclic subgroup of order $q$ is $\langle (0,p) \rangle$ and  the subgroup of order one is  $\langle (0,0) \rangle$. Let $\Omega_{pq}=\{T_{\langle(0,1)\rangle}, T_{\langle(1,1) \rangle}, \ldots , T_{\langle (p-1,1) \rangle},   T_{\langle(1,p) \rangle}\} $, $\Omega_{p}= \{  T_{\langle (0,q) \rangle} ,  T_{\langle(1,q)\rangle}, T_{\langle(2,q) \rangle}, \ldots  ,T_{\langle (p-1,q) \rangle},  T_{\langle(1,0)\rangle }\}$, $\Omega_{q}=T_{\langle (0,p)\rangle }$, and $\Omega_1= \{T_{\langle (0,0) \rangle}\}$.  By Lemma \ref{lema2.3}, each element of the set $\Omega_{pq} \cup \Omega_{p}\cup\Omega_{q} \cup \Omega_{1}$  induces a complete subgraph in $\mathscr{P}(G)$. No two elements of $\Omega_{pq}$ are adjacent because otherwise the corresponding subgroups will coincide. Similarly, no two  elements of $\Omega_{p}$   are adjacent.  An element of $\Omega_{pq}$ is adjacent to exactly one element of $\Omega_{p}$ and one element of $\Omega_{q}$, because every cyclic group has a unique subgroup of particular order.
	 Every element of  $\Omega_{pq}$ is adjacent to a distinct element of  $\Omega_{p}$.
	 % For a fixed $i$ $(0 \leq i \leq p-1)$, let  $\langle (i,q) \rangle$ be the subgroup of order $p$ of the groups $\langle (j,1) \rangle$ and $\langle (k,1) \rangle$ of orders $pq$ where $j,k \in \{0,1,2,\ldots,(p-1)\}$. So, $(jq) mod(p)  = i$ and $ (kq) mod (p) =i$. This implies that $(j-k)q= pt$ for some $t$. This is contradiction.  \\     
	% That is why permutation matrix occur in adjacency matrix.  
	 No element of  $\Omega_{p}$ is adjacent to an element of $\Omega_{q}$ because an element of order $p$ cannot be generated by an element of order $q$ and vice-versa. Every element of the set $\Omega_{pq} \cup \Omega_{p} \cup \Omega_{q}$  is adjacent to the element of $\Omega_{1}$. It is to be noted that the rows and columns of  $A(\mathscr{P}(G))$  are indexed in order by the elements of $\Omega_{pq}$, elements of $\Omega_{p}$, element of $\Omega_{q}$ and element of $\Omega_{1}$. So, $A(\mathscr{P}(G))$ is given by
	\begin{equation*}\label{equation1}
		\scriptsize	\begin{bmatrix}
			I_{(p+1) \times (p+1)} \bigotimes (J-I)_{\phi({pq})\times  \phi({pq})}& M^T _{(p+1) \times (p+1)} \bigotimes J_{\phi(pq)\times \phi(p)} & J_{(p+1)\times 1} \bigotimes J_{\phi(pq) \times \phi(q)}& J_{(p+1)\times 1} \bigotimes J_{\phi(pq)\times 1}\\
			M_{(p+1)\times (p+1)} \bigotimes J_{ \phi(p)\times \phi(pq)}& I_{(p+1)\times (p+1)} \bigotimes (J-I)_{\phi(p)\times \phi(p)}& O_{(p+1)\times 1 } \bigotimes O_{\phi(p)\times \phi(q)} & J_{(p+1)\times 1} \bigotimes J_{\phi(p)\times 1}\\
			J_{1\times (p+1)} \bigotimes J_{\phi(q)\times \phi(pq) }& O_{1\times (p+1)} \bigotimes O_{\phi(q)\times \phi(p)}&I_{1\times 1} \bigotimes (J-I)_{\phi(q)\times \phi(q)} & J_{1 \times 1} \bigotimes J_{\phi(q)\times 1}\\
			J_{1\times (p+1)} \bigotimes J_{1\times \phi(pq)} & J_{1\times (p+1)} \bigotimes J_{1\times \phi(p)}& J_{1\times 1} \bigotimes J_{1\times \phi(q)} & O_{1\times 1}
		\end{bmatrix}
	\end{equation*}
	where $M$ is a permutation matrix. 
	%Then the quotient matrix of $A(\mathscr{P}(G))$ is 
	%\begin{equation*}
	%	Q=
	%	\begin{bmatrix}
	%		(\phi(pq)-1)I_{(p+1) \times (p+1)} & \phi(p) M^T _{(p+1) \times (p+1)} & \phi(q)J_{(p+1) \times 1}& J_{(p+1) \times 1}\\
	%		\phi(pq)M_{(p+1) \times (p+1)}& (\phi(p)-1)I_{(p+1) \times (p+1)}& O_{(p+1) \times 1}&J_{(p+1) \times 1}\\
	%		\phi(pq) J_{1 \times (p+1)} & O_{1 \times (p+1)}& (\phi(q)-1)&1\\
	%		\phi(pq) J_{1 \times (p+1)} & \phi(p) J_{1 \times (p+1)}&\phi(q)&0 
	%	\end{bmatrix}
	%\end{equation*}
	
	 By Lemma \ref{lema2.5} and Theorem \ref{theorem3.1}, we  have
%	\begin{equation*}
%		\psi(A(\mathscr{P}(G),x)= \psi(Q,x) \frac{[\psi (K_{\phi(pq)},x)]^{p+1 } [\psi (K_{\phi(p)},x)]^{p+1} [\psi (K_{\phi(q)},x)]}{(x-\phi(pq)+1)^{p+1} (x-\phi(p)+1)^{p+1} (x-\phi(q)+1)}
%	\end{equation*}
%	On simplifying, we get
	\begin{equation}\label{eq4.1}
		\psi(A(\mathscr{P}(G),x)= \psi(Q,x) (1+x)^{[(p+1)(\phi(pq)+\phi(p)-2)+\phi(q)-1]}
	\end{equation}
	where 
	\begin{equation*}
		\psi(Q,x)=
		\begin{vmatrix}
			(\phi(pq)-1-x)I_{(p+1) \times (p+1)} & \phi(p) M^T _{(p+1) \times (p+1)} & \phi(q)J_{(p+1) \times 1}& J_{(p+1) \times 1}\\
			\phi(pq)M_{(p+1) \times (p+1)}& (\phi(p)-1-x)I_{(p+1) \times (p+1)}& O_{(p+1) \times 1}&J_{(p+1) \times 1}\\
			\phi(pq) J_{1 \times (p+1)} & O_{1 \times (p+1)}& (\phi(q)-1-x)&1\\
			\phi(pq) J_{1 \times (p+1)} & \phi(p) J_{1 \times (p+1)}&\phi(q)&-x 
			
		\end{vmatrix}
	\end{equation*}
	
	After  suitable  interchanging of rows and columns simultaneously (no effect on sign of determinant), we can convert $M$ to an identity matrix and this will not effect other blocks.  Then
	\begin{equation*}
			\psi(Q,x)=
		\begin{vmatrix}
			(\phi(pq)-1-x)I_{(p+1) \times (p+1)} & \phi(p) I _{(p+1) \times (p+1)} & \phi(q)J_{(p+1) \times 1}&J_{(p+1) \times 1}\\
			\phi(pq)I_{(p+1) \times (p+1)}& (\phi(p)-1-x)I_{(p+1) \times (p+1)} & O_{(p+1) \times 1}&J_{(p+1) \times 1}\\
			
			\phi(pq) J_{1 \times (p+1)} & O_{1 \times (p+1)}& (\phi(q)-1-x)&1\\
			\phi(pq) J_{1 \times (p+1)} & \phi(p) J_{1 \times (p+1)}&\phi(q)&-x
			
		\end{vmatrix}
	\end{equation*}
	
	Taking  common $\phi(pq)$ from the first $(p+1)$ columns , $\phi(p)$ from the next $(p+1)$ columns and $\phi(q)$ from the last column and taking $a=\frac{(\phi(pq)-1-x)}{\phi(pq)}$, $b= \frac{ (\phi(p)-1-x)}{\phi(p)}$, $c=\frac{(\phi(q)-1-x)}{\phi(q)}$, and $d=-x$, we have 
	\begin{equation*}
		\psi(Q,x)= {\phi(pq)}^{(p+1)} {\phi(p)}^{(p+1)} \phi(q)
		\begin{vmatrix}
			aI_{(p+1) \times (p+1)} & I _{(p+1) \times (p+1)} & J_{(p+1) \times 1}&J_{(p+1) \times 1}\\
			I_{(p+1) \times (p+1)}&bI_{(p+1) \times (p+1)} & O_{(p+1) \times 1}&J_{(p+1) \times 1}\\
			
			J_{1 \times (p+1)} & O_{1 \times (p+1)}& c&1\\
			J_{1 \times (p+1)} &  J_{1 \times (p+1)}&1&d
		\end{vmatrix}	
	\end{equation*}
	Applying the row operations $R_i \to R_i -aR_{i+p+1}$ for $1 \leq i \leq p+1$, $ R_{2p+3} \to R_{2p+3}- (R_{p+2}+ R_{p+3} + \ldots+R_{2p+2})$ and $ R_{2p+4} \to R_{2p+4}- (R_{p+2}+ R_{p+3} + \ldots+R_{2p+2})$, we get
	\begin{equation*}
		\psi(Q,x)= {\phi(pq)}^{(p+1)} {\phi(p)}^{(p+1)} \phi(q)
		\begin{vmatrix}
			O_{(p+1) \times (p+1)} &(1-ab) I _{(p+1) \times (p+1)} & J_{(p+1) \times 1}&(1-a)J_{(p+1) \times 1}\\
			I_{(p+1) \times (p+1)}&bI_{(p+1) \times (p+1)} & O_{(p+1) \times 1}&J_{(p+1) \times 1}\\
			
			O_{1 \times (p+1)}&	-bJ_{1 \times (p+1)}  &c &1-(p+1)\\
			O_{1 \times (p+1)}&(1-b)J_{1 \times (p+1)}&1&d-(p+1)
		\end{vmatrix}	
	\end{equation*}
	Now we apply the row operation on the last row $R_{2p+4} \to R_{2p+4}-R_{2p+3}$. Then we expand the determinant  along first $(p+1)$ columns successively and get
	\begin{equation*}
	\psi(Q,x)= {\phi(pq)}^{(p+1)} {\phi(p)}^{(p+1)} \phi(q) 
		\begin{vmatrix}
			(1-ab)I_{(p+1) \times (p+1)} & J_{(p+1) \times 1}&(1-a)J_{(p+1) \times 1}\\
			-bJ_{1 \times (p+1)}&c&-p\\
			J_{1 \times p+1}&1-c&d-1
		\end{vmatrix} 
	\end{equation*}
	We  write the above equation in the form: 
	\begin{equation*}
		\psi(Q,x)= {\phi(pq)}^{(p+1)} {\phi(p)}^{(p+1)} \phi(q) 
		\begin{vmatrix}
			(1-ab)I_{p \times p} & O_{p \times 1}& J_{p \times 1}&(1-a)J_{(p \times 1}\\
			O_{1 \times p}&(1-ab)&1&(1-a)\\
			-bJ_{1 \times p}&-b&c&-p\\
			J_{1 \times p}&1&1-c&d-1
		\end{vmatrix}
	\end{equation*}
	Applying the column operations $C_{i} \to C_{i}-C_{p+1}$ for $1 \leq i \leq p$, we get
	
	\begin{equation*}
		\psi(Q,x)= {\phi(pq)}^{(p+1)} {\phi(p)}^{(p+1)} \phi(q) 
		\begin{vmatrix}
			(1-ab)I_{p \times p} & O_{p \times 1}& J_{p \times 1}&(1-a)J_{p \times 1}\\
			-(1-ab)J_{1 \times p}& (1-ab)&1&1-a\\
			O_{1 \times p}& -b&c&-p\\
			O_{1 \times p}&1&1-c&d-1
		\end{vmatrix}
	\end{equation*}
	We apply the row operation $R_{p+1} \to R_1 + R_2+...+R_p+ R_{p+1}$ in the above determinant. Then
	\begin{equation*}
		\psi(Q,x)= {\phi(pq)}^{(p+1)} {\phi(p)}^{(p+1)} \phi(q) 
		\begin{vmatrix}
			(1-ab)I_{p \times p} & O_{p \times 1}& J_{p \times 1}&(1-a)J_{p \times 1}\\
			O_{1 \times p}& (1-ab)&(p+1)&(1-a)(p+1)\\
			O_{1 \times p}& -b&c&-p\\
			O_{ 1 \times p}&1&1-c&d-1
		\end{vmatrix}
	\end{equation*}
	Applying the row operations $R_{p+1} \to R_{p+1}- (1-ab)R_{p+3}$, $R_{p+2} \to R_{p+2}+ bR_{p+3}$ and by taking common $(1-ab)$ from first  $p$ columns, we get 
	\begin{equation*}
		\psi(Q,x)=\scriptsize {\phi(pq)}^{(p+1)} {\phi(p)}^{(p+1)} \phi(q) (1-ab)^p
		\begin{vmatrix}
			I_{p \times p} & O_{p \times 1}& J_{p \times 1}&(1-a)J_{p \times 1}\\
			O_{1 \times p}& 0&p+c+ab(1-c)&(1-a)(p+1)-(1-ab)(d-1)\\
			O_{1 \times p}& 0&c+b(1-c)&-p+b(d-1)\\
			O_{ 1 \times p}&1&1-c&d-1
		\end{vmatrix}
	\end{equation*}
	Expanding the above determinant along the first $(p+1)$ columns and then applying the row operation $R_1 \to R_1 -aR_2$ to the resulting determinant, we have 
	\begin{equation*} 
		\psi(Q,x) = 
		{\phi(pq)}^{(p+1)} {\phi(p)}^{(p+1)} \phi(q) 	(1-ab)^p det(C)
	\end{equation*}
	where
	\begin{equation*}
		det(C)= \begin{vmatrix}
			p+c-ac& p+(1-a)+(1-d)\\
			b+c-bc& -p-b(1-d)
		\end{vmatrix}
	\end{equation*}
	Thus, $\psi(Q,x)= 	{\phi(pq)}^{(p+1)} {\phi(p)}^{(p+1)} \phi(q)	(1-ab)^p[(p+c-ac)(-p-b(1-d))-(b+c-bc)(p+(1-a)+(1-d))]$. 
	%for suitable integer $t$. By putting the value of $a$, $b$, $c$, the coefficient of maximum degree $x$ in $|Q-xI|$ is $(-1)^{t+p+3}$. But the coefficient of maximum degree $x$ in $|Q-xI|$ is $2p+5$. Then we can set $t=p+2$, as required.
	The factor $(1-ab)^p= \big(\frac{(1+x)(\phi(pq)+\phi(p)-1-x)}{\phi(pq)\phi(p)}\big)^p$ gives the  eigenvalues $ -1$ and $\phi(pq)+\phi(p)-1$ with both multiplicities $p$. And   ${\phi(pq)}\phi(p)\phi(q){det(C)}$ is the quartic polynomial given in the hypothesis and from this factor we get four eigenvalues. 
	Finally by Equation (\ref{eq4.1}), we get multiplicity of $-1$ is $p^2q-p-4$.
\end{proof}

\begin{example}{\rm
		We illustrate Theorem \ref{theorem2.2} for $p=3$ and $q=2$. So we get the partition of   $V(\mathscr{P}(\mathbb{Z}_3 \times \mathbb{Z}_6))$  as 
		
		$	V(\mathscr{P}(\mathbb{Z}_3 \times \mathbb{Z}_6))= T_{ (0,1) } \cup T_{ (1,1)} \cup T_{ (2,1)} \cup T_{ (1,3)} \cup T_{ (0,2) } \cup T_{ (1,2)} \cup T_{ (2,2)} \cup T_{ (1,0)} \cup T_{ (0,3) } \cup T_{ (0,0)}$
	
		By indexing the rows and columns in the order of the above partition, we get
		\begin{equation*}
			A(\mathscr{P}(\mathbb{Z}_3 \times \mathbb{Z}_6))=
			\begin{bmatrix}
				I_{4} \bigotimes (J-I)_{2 \times 2}& M^T _{4} \bigotimes J_{2 \times 2} & J_{4 \times 1} \bigotimes J_{2 \times 1}& J_{4 \times 1} \bigotimes J_{2 \times 1}\\
				M_{4} \bigotimes J_{ 2 \times 2}& I_{4} \bigotimes (J-I)_{2 \times 2}& O_{4 \times 1 } \bigotimes O_{2 \times 1} & J_{4 \times 1} \bigotimes J_{2 \times 1}\\
				J_{1 \times 4} \bigotimes J_{1 \times 2 }& O_{1 \times 4} \bigotimes O_{1 \times 2}&I_1 \bigotimes (J-I)_{1 \times 1} & J_{1 \times 1} \bigotimes J_{1 \times 1}\\
				J_{1 \times 4} \bigotimes J_{1 \times 2} & J_{1 \times 4} \bigotimes J_{1 \times 2}& J_{1 \times 1} \bigotimes J_{1 \times 1} & O_{1\times 1}
				
			\end{bmatrix}
		\end{equation*}
		where $M$ is a permutation matrix 
		\begin{equation*}
			M= \begin{bmatrix}
				1 & 0 & 0 & 0\\
				0 & 0 & 1 & 0\\
				0 & 1 & 0 & 0\\
				0& 0 & 0& 1
			\end{bmatrix}
		\end{equation*}
		
		Then by Theorem $\ref{theorem3.1}$, we get 	$	\psi(A(\mathscr{P}(\mathbb{Z}_3 \times \mathbb{Z}_{6}),x)= \psi(Q,x)~ (1+x)^{8}$
		
		where $\psi(Q,x)$ is given by
		\begin{equation*}
			 \begin{bmatrix}
				(1-x)I_4& 2M_4 ^T&J_{4 \times 1}&J_{4 \times 1}\\
				2M_4 &(1-x)I_4&O_{4 \times 1}&J_{4 \times 1}\\
				2J_{1 \times 4}&O_{1 \times 4}& -x&1\\
				2J_{1 \times 4}&2J_{1 \times 4}&1&-x
			\end{bmatrix}
		\end{equation*} 
		After interchanging the  rows $R_6$ and $R_7$, and columns $C_6$ and $C_7$ simultaneously,we get
		\begin{equation*}
			\psi(Q,x)= \begin{vmatrix}
				(1-x)I_4& 2I_4&J_{4 \times 1}&J_{4 \times 1}\\
				2I_4&(1-x)I_4&O_{4 \times 1}&J_{4 \times 1}\\
				2J_{1 \times 4}&O_{1 \times 4}& -x&1\\
				2J_{1 \times 4}&2J_{1 \times 4}&1&-x
			\end{vmatrix} =  (2)^{4}(2)^4\begin{vmatrix}
			\frac{(1-x)}{2}	I_4& I_4&J_{4 \times 1}&J_{4 \times 1}\\
			I_4&\frac{(1-x)}{2}I_4&O_{4 \times 1}&J_{4 \times 1}\\
			J_{1 \times 4}&O_{1 \times 4}& -x&1\\
			J_{1 \times 4}&J_{1 \times 4}&1&-x
		\end{vmatrix}
		\end{equation*}

		We put $a=\frac{(1-x)}{2}$, $b=-x$, apply the row operations $R_i \to R_i -aR_{i+4}$ for $1\leq i\leq4$, $R_{10} \to  R_{10}-R_9$, and  get 
		\begin{equation*}
		\psi(Q,x)= (2)^{8}\begin{vmatrix}
				O_4& (1-a^2)I_4&J_{4 \times 1}&(1-a)J_{4 \times 1}\\
				I_4&aI_4&O_{4 \times 1}&J_{4 \times 1}\\
				J_{1 \times 4}&O_{1 \times 4}&b&1\\
				O_{1 \times 4}&J_{1 \times 4}&1-b&b-1
			\end{vmatrix}
		\end{equation*}
		Now applying the row operation $R_9 \to R_9 -(R_5+R_6+R_7+R_8)$, and expanding along the first four columns, we get
		\begin{equation*}
			\psi(Q,x)|= (2)^{8}\begin{vmatrix}
				(1-a^2)I_4&J_{4 \times 1}&(1-a)J_{4 \times 1}\\
				-aJ_{1 \times 4}&b&-3\\
				J_{1 \times 4}&1-b&b-1
			\end{vmatrix}
		\end{equation*}
		First, applying the column operations $C_i \to C_i -C_4$, $1\leq i \leq 3$, then row operation $R_4 \to R_4 +R_3+R_2+R_1$, and  expanding along first three columns, we have
		\begin{equation*}
			\psi(Q,x)= 2^8 (1-a^2)^3\begin{vmatrix}
				(1-a^2)& 4&4(1-a)\\
				-a&b&-3&\\
				1&1-b&b-1
			\end{vmatrix}
		\end{equation*} 
		After putting the values of $a$ and $b$, we get $	|Q-xI|= (1+x)^3(3-x)^3(x+3)(x-1)(x^2-4x-17)$.
		So, the adjacency spectrum of $\mathscr{P}(\mathbb{Z}_3 \times \mathbb{Z}_{6})$ is 
		\begin{equation*}
			\begin{pmatrix}
				-3&-1&2-\sqrt{21}&1&3&2+\sqrt{21}\\
				1&11&1&1&3&1		
			\end{pmatrix}.
	\end{equation*}}

\end{example}

\begin{theorem}
	Let $G= \mathbb{Z}_{p^2} \times \mathbb{Z}_{p^2}$, where $p$ is a prime number. The spectrum of $\mathscr{P}(G)$ consists of: 
	%	\begin{equation*}
		%		\begin{pmatrix}
			%			-1& p^2-p-1& \alpha_1& \alpha_2 &x_1 &x_2&x_3\\
			%			p^4 -p^2 -2p -2 & p^2-1& p& p& 1& 1& 1
			%			
			%		\end{pmatrix}
		%	 	\end{equation*}
	\begin{enumerate}
		\item $-1$ with multiplicity $p^4-p^2-2p-2$.
		\item $p^2-p-1$ with multiplicity $p^2-1$.
		\item  $\frac{1}{2} \big(p^2 \pm (p-1)\sqrt{5p^2-2p+1}-3\big)$ with multiplicity $p$.
		\item And  the roots of the polynomial $(x^3+(3-p^2)x^2 +(-2p^4+3p^3-4p^2 +p+3)x-p^5+p^4-2p^2+p+1)$.
	\end{enumerate}
	%	where $x_1, x_2,x_3$ are, and $\alpha_1= \frac{1}{2} \big(p^2-(p-1)\sqrt{5p^2-2p+1}-3\big)$, $\alpha_2=\frac{1}{2} \big(p^2+(p-1)\sqrt{5p^2-2p+1}-3\big).$
\end{theorem}
\begin{proof}
	For $G= \mathbb{Z}_{p^2} \times \mathbb{Z}_{p^2}$, an element of the group $G$ is of order $1$, $p$ or $p^2$.  We get that $G$ has exactly $  \frac{\phi(p^2)(2+2\phi(p)+\phi(p^2))}{\phi(p^2)}= p(p+1)$ cyclic subgroups of order $p^2$,  $ \frac{(2\phi(1)+\phi(p))\phi(p) }{\phi(p)}= p+1$ cyclic subgroups of order $p$, and one subgroup of order $1$.\\
	By Case 2 of the proof of Lemma \ref{lemma3.3}, the cyclic subgroups of order $p^2$ are $\langle(0,1)\rangle, \langle(1,1)\rangle,\langle(2,1)\rangle,$ $\ldots,\langle(p^2-1,1)\rangle,$ $\langle {(1,p)} \rangle,$ $ \langle(2,p)\rangle,$ $ \ldots, \langle(p-1,p)\rangle,$ $ \langle(1,0)\rangle$. By Case 3 of the proof of Lemma \ref{lemma3.3}, the cyclic subgroups of order $p$ are $\langle(0,p)\rangle,\langle(p,p)\rangle,$ $\langle(2p,p)\rangle,$ $\ldots,\langle((p-1)p,p)\rangle,\langle(p,0)\rangle$ and  by Case 4 of the same  lemma, the subgroup of order $1$ is $\langle(0,0)\rangle$. Let $\Omega_{p^2}=  T_{\langle(0,1)\rangle}, T_{\langle(1,1)\rangle}, T_{\langle(2,1)\rangle}, \ldots, T_{\langle(p^2-1,1)\rangle}, T_{\langle(1,p)\rangle},$ $ T_{\langle(2,p)\rangle}, \ldots ,  T_{\langle(p-1,p)\rangle}, T_{\langle(1,0)\rangle} $, $\Omega_{p}=  T_{\langle(0,p)\rangle}, T_{\langle(p,p)\rangle}, T_{\langle(2p,p)\rangle}, \ldots, T_{\langle(p(p-1),p)\rangle},   T_{\langle(p,0)\rangle}$ and $\Omega_{1}=	T_{\langle(0,0)\rangle}$.  No two elements of  $\Omega_{p^2}$ are  adjacent because otherwise the corresponding subgroups will coincide. Similarly, no two elements of $\Omega_{p}$ are adjacent.  Each element of  $\Omega_{p^2}$ is adjacent to exactly one element of  $\Omega_{p}$  because every cyclic group has a unique subgroup of particular order.
    %Clearly, $\langle (jp,p)\rangle$ is subgroup of order $p$ of the group $\langle (ip+j,1)\rangle$ for $i,j \in \{0,1,2,\ldots,p-1\}$. And $\langle(p,0)\rangle$ is subgroup of order $p$ of the groups $\langle(i,p)\rangle$ and  $\langle(1,0)\rangle$.  So for each $j$, $\langle (jp,p)\rangle$ is the subgroup of exactly $(p+1)$ cyclic groups of order $p^2$.
	 %Therefore, the block matrix induced by elements of $\Omega_{p^2}$ and $\Omega_{p}$ is  $J_{1 \times p} \bigotimes \big(M_{(p+1) \times (p+1)} \bigotimes J_{\phi(p) \times \phi(p^2)} \big)^T$.
	  And every element from the set  $\Omega_{p^2}\cup \Omega_{p}$  is adjacent to the element of  $\Omega_{1}$.   It is to be noted that the rows and columns of  $A(\mathscr{P}(G))$  are indexed in order by the elements of $\Omega_{1}$, elements of $\Omega_{p^2}$, and element of $\Omega_{p}$. So, $A(\mathscr{P}(G))$ is given by
	\begin{equation*}\label{equation17}
		A(\mathscr{P}(G))= \scriptsize	\begin{bmatrix}
			0& J_{1 \times (p^2+p)} \bigotimes J_{1 \times \phi(p^2)}& J_{1 \times (p+1)} \bigotimes J_{1 \times \phi(p)}\\
			J_{(p^2+p)  \times 1} \bigotimes J_{\phi(p^2) \times 1}& I_{(p^2+p) \times (p^2+p)} \bigotimes (J-I)_{\phi({p^2}) \times \phi({p^2})}& J_{p \times 1} \bigotimes \big(M_{(p+1) \times (p+1)} \bigotimes J_{\phi(p^2) \times \phi(p)} \big)\\
			J_{(p+1) \times 1} \bigotimes J_{\phi(p) \times 1} &  J_{1 \times p} \bigotimes \big(M_{(p+1) \times (p+1)} \bigotimes J_{\phi(p) \times \phi(p^2)} \big)^T& I_{(p+1) \times (p+1) } \bigotimes (J-I)_{\phi(p) \times \phi(p)}\\
		\end{bmatrix}
	\end{equation*}
	where $M$ is a permutation matrix. 
	%Then the quotient matrix of $A(\mathscr{P}(G))$ is 
	%\begin{equation*}
	%	Q=
	%	\begin{bmatrix}
	%		0 & \phi(p^2)  J_{1\times (p^2+p) } & \phi(p) J_{1 \times (p+1)}\\
	%		J_{(p^2+p) \times 1} &(\phi(p^2)-1) I_{(p^2+p) \times (p^2+p)}& J_{p \times 1} \bigotimes \big(\phi(p) M_{(p+1) \times (p+1)}\big) &\\
	%		
	%		J_{(p+1) \times 1} & J_{1 \times p} \bigotimes \big(\phi(p^2) M_{(p+1)\times (p+1)} ^T\big)& (\phi(p)-1) I_{(p+1) \times (p+1)}
	%		
	%	\end{bmatrix}
	%\end{equation*}
	
	Then  by Lemma \ref{lema2.5}  and Theorem $\ref{theorem3}$, we get 
	\begin{equation*}
		\psi(A(\mathscr{P}(G)),x)= (1+x)^{(\phi(p^2)-1)(p^2+p)+(\phi(p)-1)(p+1)}~ \psi(Q,x)
	\end{equation*}
where
	\begin{equation*}
	\psi(Q,x)=
		\begin{vmatrix}
			-x & \phi(p^2)  J_{1\times (p^2+p) } & \phi(p) J_{1 \times (p+1)}\\
			J_{(p^2+p) \times 1} &(\phi(p^2)-1-x) I_{(p^2+p) \times (p^2+p)}& J_{p \times 1} \bigotimes \big(\phi(p) M_{(p+1) \times (p+1)}\big) &\\
			
			J_{(p+1) \times 1} & J_{1 \times p} \bigotimes \big(\phi(p^2) M_{(p+1)\times (p+1)} ^T\big)& (\phi(p)-1-x) I_{(p+1) \times (p+1)}
			
		\end{vmatrix}
	\end{equation*}
	
	After  suitable  interchanging of rows and columns simultaneously (no effect on sign of determinant), we can convert $M$ to an identity matrix and this will not effect other blocks. Taking common $1$ from first column, $\phi(p^2)$ from next $(p^2+p)$ columns and $\phi(q)$ from last $(p+1)$ columns, we get

	\begin{equation*}
		\psi(Q,x)= \phi(p^2)^{p^2+p}\phi(p)^{p+1}\begin{vmatrix}
			-x &   J_{1\times (p^2+p) } &  J_{1 \times (p+1)}\\
			J_{(p^2+p) \times 1} &\frac{(\phi(p^2)-1-x)}{\phi(p^2)} I_{(p^2+p) \times (p^2+p)}& J_{p \times 1} \bigotimes \big( I_{(p+1) \times (p+1)}\big) &\\
			
			J_{(p+1) \times 1} & J_{1 \times p} \bigotimes \big( I_{(p+1)\times (p+1)}\big)& \frac{(\phi(p)-1-x)}{\phi(p)} I_{(p+1) \times (p+1)}
		\end{vmatrix}
	\end{equation*}
	By taking $\frac{(\phi(p^2)-1-x)}{\phi(p^2)}= a$, $\frac{(\phi(p)-1-x)}{\phi(p)}=b$, and $-x = c$, we have
	
	\begin{equation*}
		\psi(Q,x)= \phi(p^2)^{p^2+p}\phi(p)^{p+1}\begin{vmatrix}
			c&   J_{1\times (p^2+p) } &  J_{1 \times (p+1)}\\
			J_{(p^2+p) \times 1} &a I_{(p^2+p) \times (p^2+p)}& J_{p \times 1} \bigotimes \big( I_{(p+1) \times (p+1)}\big) &\\
			
			J_{(p+1) \times 1} & J_{1 \times p} \bigotimes \big( I_{(p+1)\times (p+1)}\big)& b I_{(p+1) \times (p+1)}
		\end{vmatrix}
	\end{equation*}
	By applying the row operations on blocks $R_{i} \to R_{i}-R_{(p+1)}$ for $2 \leq i \leq p$, we get 
	\begin{equation*}
		\psi(Q,x)=\scriptsize  \phi(p^2)^{p^2+p}\phi(p)^{p+1} \begin{vmatrix}
			c & J_{1 \times (p+1)}& J_{1 \times (p+1)}&\cdots & J_{1 \times (p+1)}& J_{1 \times (p+1)}\\
			O_{(p+1)\times (p+1)}& aI_{(p+1)\times (p+1)} &O_{(p+1)\times (p+1)}& \cdots&-aI_{(p+1)\times (p+1)} & O_{(p+1)\times (p+1)}\\
			O_{(p+1)\times (p+1)}&O_{(p+1)\times (p+1)}& aI_{(p+1)\times (p+1)}  &  \cdots&-aI_{(p+1)\times (p+1)} & O_{(p+1)\times (p+1)}\\
			\vdots &\vdots&\vdots & \ddots& \vdots &\vdots\\ 
			J_{(p+1) \times 1}&O_{(p+1)\times (p+1)}& O_{(p+1)\times (p+1)}& \cdots&aI_{(p+1)\times (p+1)} &I_{(p+1)\times (p+1)} \\
			J_{(p+1) \times 1}&I_{(p+1)\times (p+1)} &I_{(p+1)\times (p+1)} &\cdots&I_{(p+1)\times (p+1)} &bI_{(p+1)\times (p+1)} 
			
		\end{vmatrix}	
	\end{equation*}
	Applying the column operation on $(p+1)^{th}$ block, $C_{p+1} \to C_{2}+C_{3}+\ldots+C_{p+1}$, we get
	\begin{equation*}
	\psi(Q,x)= \scriptsize  \phi(p^2)^{p^2+p}\phi(p)^{p+1} \begin{vmatrix}
			c & J_{1 \times (p+1)}& J_{1 \times (p+1)}&\cdots &p J_{1 \times (p+1)}& J_{1 \times (p+1)}\\
			O& aI &O& \cdots&O& O\\
			O&O& aI &  \cdots&O& O\\
			\vdots &\vdots&\vdots & \ddots& \vdots &\vdots\\ 
			J_{(p+1) \times 1}&O& O& \cdots&aI&I\\
			J_{(p+1) \times 1}& I&I&\cdots&pI&bI
			
		\end{vmatrix}	
	\end{equation*}
	Expanding along the rows $R_2,R_3,\ldots, R_{p^2}$, we get
	\begin{equation*}
	\psi(Q,x) =\phi(p^2)^{p^2+p}\phi(p)^{p+1} a^{(p+1)(p-1)} \begin{vmatrix}
			c& p J_{1 \times (p+1)}&  J_{1 \times (p+1)}\\
			J_{(p+1) \times 1} & aI_{(p+1)\times (p+1)}& I_{(p+1)\times (p+1)}\\
			J_{(p+1) \times 1}&pI_{(p+1)\times (p+1)}&bI_{(p+1)\times (p+1)}
			
		\end{vmatrix}
	\end{equation*}
	
	By applying the row operations $R_1 \to R_{1}-(R_2+R_3+\ldots+R_{p+2})$, $R_{p+2
		+i} \to R_{p+2+i}-bR_{i+1}$ for $1 \leq i \leq p+1$, we get
	\begin{equation*}
		\psi(Q,x) =\phi(p^2)^{p^2+p}\phi(p)^{p+1}	a^{(p^2-1)}
		\begin{vmatrix}
			c-(p+1)&(p-a)J_{(p+1) \times 1}& O_{1 \times (p+1)}\\
			J_{(p+1) \times 1}&aI_{(p+1)\times (p+1)}&I_{(p+1)\times (p+1)}\\
			(1-b)	J_{(p+1) \times 1} & (p-ab)I_{(p+1)\times (p+1)}& O_{(p+1)\times (p+1)} 
			
		\end{vmatrix}
	\end{equation*}
	Expanding along the last $p+1$ columns, we get
	\begin{equation*}
	\psi(Q,x) =\phi(p^2)^{p^2+p}\phi(p)^{p+1}	a^{(p^2-1)}
		\begin{vmatrix}
			c-(p+1)&(p-a)J_{(p+1) \times 1}\\
			(1-b)J_{(p+1) \times 1}&  (p-ab)I_{(p+1)\times (p+1)}
			
		\end{vmatrix}
	\end{equation*}
	Applying the  row operations $R_{i} \to R_{i}-R_2$ for $3 \leq i \leq p+2$ and then the column operation $C_2 \to C_2 +C_3+\ldots+C_{p+2}$, we get 
	\begin{equation*}
		\psi(Q,x) =\phi(p^2)^{p^2+p}\phi(p)^{p+1}	a^{(p^2-1)}
		\begin{vmatrix}
			c-(p+1)&(p+1)(p-a)&(p-a)&\cdots & (p-a)\\
			
			(1-b)&	p-ab&0&\cdots & 0\\
			0&0& p-ab&\cdots&0\\
			\vdots & \vdots &\vdots& \ddots & \vdots\\
			0&0&0& \cdots&  p-ab
			
		\end{vmatrix}
	\end{equation*}
This implies
	\begin{equation}\label{eq19}
\psi(Q,x) =	\phi(p^2)^{p^2+p}\phi(p)^{p+1}	a^{(p^2-1)}(p-ab)^p[(p-ab)(c-(p+1))-(p+1)(1-b)(p-a)]
	\end{equation}
	
	% for suitable integer $t$. By putting the value of $a$, $b$ the coefficient of maximum degree $x$ in $|Q-xI|$ is $(-1)^{t+p^2+p-1}$. But the coefficient of maximum degree $x$ in $|Q-xI|$ is $p^2+2p+2$. Then we can set $t=p+3$, as required.
	Here, $(p-ab)$ is a quadratic polynomial and  $ [(p-ab)(c-(p+1))-(p+1)(1-b)(p-a)]$ is a cubic polynomial. By putting values of $a$, $b$ and $c$ in  Equation (\ref{eq19}), then $\psi(Q,x)$ is \\
	$ \phi(p^2)^{p^2+p}\phi(p)^{p+1} \big(\phi(p^2)-1-x\big)^{(p^2-1)} \big(-x^2 +(p^2-3)x+p^4-3p^3+4p^2 -p-2 \big)^p$ $[x^3+(3-p^2)x^2 +(-2p^4+3p^3-4p^2 +p+3)x-p^5+p^4-2p^2+p+1]$. 
\end{proof}

\section{Characteristic polynomial of $\mathscr{P}(\mathbb{Z}_{m} \times \mathbb{Z}_{n}), m \nmid n$}\label{5}
 If $m \nmid n$ and $gcd(m,n)=1$, then $\mathbb{Z}_m \times \mathbb{Z}_n \cong \mathbb{Z}_{mn}$. So here we study the spectrum of $\mathscr{P}(\mathbb{Z}_m \times \mathbb{Z}_n)$ when $m \nmid n$ and $gcd(m,n) \neq 1$. Recall that we have taken the set of all divisors of $n$ as  $S= \{d_1,d_2,\ldots, d_k\}$. Let $S'=\{\xi_1,\xi_2,\ldots,\xi_t,\xi_{t+1},\ldots,\xi_s\}$ be the set of divisors of $m$. Since $gcd(m,n) \neq 1$, without loss of generality let  $S \cap S' = $ $\{d_{r+1},$ $d_{r+2},\ldots,d_{k=r+h}\}$ $=\{\xi_{t+1},\xi_{t+2}, \ldots,\xi_{s= t+h}\}$ for some positive integers $r,t$  and $h$. Let $T= \big\{ d_i \xi_j: gcd(d_i,\xi_j)=1, 1 \leq i \leq k, 1 \leq j \leq s\big\}$, and $T' = T / \{S \cup S'\}$. 
\begin{theorem}\label{theorem5.1}
	Let $ G= \mathbb{Z}_m \times \mathbb{Z}_n$ with $m \nmid n$ and  $gcd(m,n) \neq 1$.  Then 
	\begin{equation*}
		\psi(A(\mathscr{P}(G)),x) = (1+x)^{\alpha} \psi(Q,x)
	\end{equation*}
	where $\alpha = \sum_{i=1} ^k n_i(\phi(d_i)-1) + \sum_{j=1} ^t m_j(\phi(\xi_j)-1) + \sum_{a \in T'} \eta_a(\phi(a)-1)$, $n_i, m_j, \eta_a$ are respectively the number of cyclic subgroups of $G$ of order $d_i, \xi_j, a $,  and $Q =(q_{ij})_{l \times l}$ with
	\begin{equation*} 
		q_{ij}	= \begin{cases}
			\phi(|\langle(a_j,b_j)\rangle|) & \text{if~} i \neq j \text{~and~} T_{\langle (a_i,b_i) \rangle } \sim T_{\langle (a_j,b_j) \rangle} \\
			\phi(|\langle(a_i,b_i)\rangle|)-1 &  \text{if~}  i = j \\
			0 & \text{otherwise}\\
		\end{cases}
	\end{equation*}
\end{theorem}
\begin{proof}
	We obtain that $S \cup S' \cup T'$ is the set of all possible orders of the cyclic subgroups of $G$. Then we proceed exactly in the similar way as in the proof of Theorem \ref{theorem3} and get the required result,
	% with
%\begin{equation*} 
%	q_{ij}	= \begin{cases}
%		\phi(|\langle(a_j,b_j)\rangle|) & \text{if~} i \neq j \text{~and~} T_{\langle (a_i,b_i) \rangle } \sim T_{\langle (a_j,b_j) \rangle} \\
%		\phi(|\langle(a_i,b_i)\rangle|)-1 &  \text{if~}  i = j \\
%		0 & \text{otherwise}\\
%	\end{cases}
%\end{equation*}
\end{proof}
\begin{lemma}\label{lemma5.1}
		Let $G= \mathbb{Z}_{p^2} \times \mathbb{Z}_n$ with $p|n$ and $p^2 \nmid n$. Let the first $c $ elements in $S$ be relatively prime to $p^2$, and the next $k-c$  elements in $S$ be divisible by $p$. Then the cyclic subgroups of $G$ of different orders are as given below: 
		the cyclic subgroups of order $d_i$ are
		${\langle(0,\frac{n}{d_i})\rangle} $  for $1 \leq i \leq c$; $ {\langle(0,\frac{n}{d_i})\rangle} ,{\langle(p,\frac{n}{d_i})\rangle}, {\langle(2p,\frac{n}{d_i})\rangle}, \ldots,$ $ {\langle(p(p-1),\frac{n}{d_i})\rangle},$ ${\langle(p,\frac{np}{d_i})\rangle}$   for $c+1 \leq i \leq k$; and the cyclic subgroups of order $d_ip^2$ are ${\langle(1,\frac{n}{d_i})\rangle},{\langle(1,\frac{n}{d_ip})\rangle},$ $ {\langle(2,\frac{n}{d_ip})\rangle}, \ldots,$ $ {\langle(p-1,\frac{n}{d_ip})\rangle}  $ for $ 1 \leq i \leq c$. 
\end{lemma}
\begin{proof}
%	As $p|n$ but $p^2 \nmid n$, the possible order of the elements of the group $G$ are divisors of $n$ and as well as some non-divisors of $n$.
	We have at least one $d_i$ which is relatively prime to $p^2$. So $c \geq 1$. Since $p|n$, there exists at least one $d_i$ such that $p|d_i$. So $ k-c \geq 1$.
	Then the order of the elements of $G$ are $d_1,d_2, \ldots,d_k,d_1p^2,d_2p^2,\ldots,d_cp^2.$  %Because the possible orders of the elements of the group $\mathbb{Z}_{p^2}$ are $1,p,p^2$ and the possible orders of the elements of the group $\mathbb{Z}_n$ are $d_1,d_2, \ldots , d_k$.
	 Now the cyclic subgroups of $G$ of different orders are obtained as below.\\
	\textbf{Case:1} For each $i$, $1 \leq i \leq c$,  the number of cyclic subgroup of $G$ of order $d_i$  is $\frac{(\phi(1)\phi(d_i))}{\phi(d_i)}=1$. This cyclic subgroup is ${\langle(0,\frac{n}{d_i})\rangle} $  for  $1 \leq i \leq c$.\\ 
	\textbf{Case:2}  For each $i$, $c+1\leq i \leq k$,  the number of cyclic subgroups of group $G$ of order $d_i$   are $\frac{\big[\phi(1)\phi(d_i)+\phi(p)\phi(d_i)+\phi(p)\phi(\frac{d_i}{p})\big]}{\phi(d_i)}=(p+1)$. These $(p+1)$ cyclic subgroups are $ {\langle(0,\frac{n}{d_i})\rangle} ,$ ${\langle(p,\frac{n}{d_i})\rangle},$ $ {\langle(2p,\frac{n}{d_i})\rangle},$ $ \ldots, {\langle(p(p-1),\frac{n}{d_i})\rangle},{\langle(p,\frac{np}{d_i})\rangle}$   for  $c+1 \leq i \leq k$. \\ 
	\textbf{Case:3}  For  each $i$, $1 \leq i \leq c$, the number of cyclic subgroups of group $G$ of order $d_ip^2$  are $\frac{\phi(p^2)\phi(d_i)+\phi(p^2)\phi(d_ip)}{\phi(d_i)}= p$. These $p$ cyclic subgroups are  ${\langle(1,\frac{n}{d_i})\rangle},{\langle(1,\frac{n}{d_ip})\rangle}, {\langle(2,\frac{n}{d_ip})\rangle}, \ldots, {\langle(p-1,\frac{n}{d_ip})\rangle}  $  for  $1 \leq i \leq c$.
\end{proof}

 By Lemma \ref{lemma5.1}, there are total $c+(p+1)(k-c)+pc= l(say)$ cyclic subgroups of $\mathbb{Z}_{p^2} \times \mathbb{Z}_n$. We list all these in the order of their occurrence in  the statement of Lemma \ref{lemma5.1}, that is $S_1 =  \langle(0,\frac{n}{d_1})\rangle$, $S_2 = \langle(0,\frac{n}{d_2})\rangle,\ldots,S_{l}=\langle(p-1,\frac{n}{d_{kp}})\rangle$. Without loss of generality, let $d_c=1$. $S'= \{p^2,p,1\}$ and $T'=\{d_1p^2,d_2p^2,\ldots,d_{c-1}p^2\}$.

\begin{theorem}\label{theorem5.2}
	Let $G= \mathbb{Z}_{p^2} \times \mathbb{Z}_n$ with $p|n$ and $p^2 \nmid n$. Let the first $c$ elements in $S$ be relatively prime to $p^2$, and the next $k-c$  elements in $S$ be divisible by $p$. Then,
	\begin{equation*}
		\psi(A(\mathscr{P}(G)),x)= (1+x)^{\alpha} ~\psi(Q,x)
	\end{equation*}
	where $\alpha = {\sum_{i=1} ^{c} (\phi(d_i)-1)+ (p+1) \sum_{i=c+1} ^{k} (\phi(d_i)-1) +p \sum_{i=1} ^{c} (\phi(d_ip^2)-1) }$ and $Q =(q_{ij})_{l \times l}$ with
	\begin{equation*} 
		q_{ij}= \begin{cases}
			\phi(|S_j|) & \text{if~} i \neq j \text{~and~} T_{S_i} \sim T_{S_j} \\
			\phi(|S_i|)-1 & \text{if~}  i = j\\
			0 & \text{otherwise}\\
		\end{cases}
	\end{equation*} 
\end{theorem}
\begin{proof}   
	We have $V(\mathscr{P}(G))= T_{S_1} \cup T_{S_2}\cup \ldots \cup T_{S_l}$. 	By Lemma \ref{lemma5.1}, number of cyclic subgroups of order $d_i$ is $ 1$ for $1 \leq i \leq c$; $ (p+1)$ for $ c+1 \leq i \leq k$.  The  number of cyclic subgroups of order $p^2$ are $p$. And the number of cyclic subgroups of order $d_ip^2$ is $ p$ for $1 \leq i \leq c-1$. Then by Lemma \ref{lema2.5} and Theorem  \ref{theorem5.1},  we get the required result.
\end{proof}

\begin{theorem}
	Let $G= \mathbb{Z}_{p^2} \times \mathbb{Z}_{pq}$, where $p,q$ are distinct prime numbers. Then the spectrum of $\mathscr{P}(G)$ consists of:
	\begin{enumerate}
	\item $-1$ with multiplicity $p^3q-2p-6$.
	\item $p^2q-pq-1$ with multiplicity $p-1$.
	\item $pq-q-1$ with multiplicity $p-1$.
	\item And the remaining $8$ eigenvalues are the roots of the determinant
	\begin{equation}\label{eq6}
	\begin{vmatrix}
		1-ab& 0&-b&1-b&-b&1-b\\
		1&0&c&1&1&1\\
		0&1-cd&0&0&-d&1-d\\
		1-a&0&1-p&d-p&-p&1-p\\
		1&1&1&0&e&1\\
		1-a&1-c&1-p&1-p&1-2p&f-2p
	\end{vmatrix}
\end{equation}
	\end{enumerate}
where $ a= \frac{\phi(p^2q)-1-x}{\phi(p^2q)}$, $b =\frac{\phi(p^2)-1-x}{\phi(p^2)} $, $c=\frac{\phi(pq)-1-x}{\phi(pq)}$, $d= \frac{\phi(p)-1-x}{\phi(p)}$, $e=\frac{\phi(q)-1-x}{\phi(q)}$ and $f=-x$.
\end{theorem}
\begin{proof}
	The possible orders of the elements of $G$ are $1,q,p,pq,p^2$ and $p^2q$. By Lemma \ref{lemma5.1}, the cyclic subgroups of $G$ of different orders are as follow:  order $p^2q$ are $\langle (1,1)\rangle,$ $ \langle(2,1)\rangle,$ $ \ldots, \langle((p-1),1)\rangle,\langle(1,p)\rangle$;  order $p^2$ are  $\langle (1,q)\rangle,$ $ \langle(2,q)\rangle,$ $ \ldots, \langle((p-1),q)\rangle,$ $\langle(1,0)\rangle$;  order $pq$ are $\langle (0,1)\rangle,$ $ \langle(p,1)\rangle,$ $ \langle(2p,1)\rangle, \ldots,$ $ \langle((p-1)p,1)\rangle,$ $\langle(p,p)\rangle$; order $p$ are $\langle (0,q)\rangle,$ $ \langle(p,q)\rangle,$ $ \langle(2p,q)\rangle, \ldots,$ $ \langle((p-1)p,q)\rangle,\langle(p,0)\rangle$; order $q$ is $\langle (0,p)\rangle$, and the  order $1$ is $\langle(0,0)\rangle$. Let  $\Omega_{p^2q} =\{ T_{\langle (1,1)\rangle},$ $ T_{\langle(2,1)\rangle},$ $ \ldots, T_{\langle((p-1),1)\rangle},T_{\langle(1,p)\rangle}\}$; $\Omega_{p^2}=\{ T_{\langle (1,q)\rangle},$ $ T_{\langle(2,q)\rangle},$ $ \ldots, T_{\langle((p-1),q)\rangle},$ $T_{\langle(1,0)\rangle}\}$; $\Omega_{pq}=\{ T_{\langle (0,1)\rangle},$ $ T_{\langle(p,1)\rangle},$ $ T_{\langle(2p,1)\rangle}, \ldots,$ $T_{ \langle((p-1)p,1)\rangle},$ $T_{\langle(p,p)\rangle}\}$; $\Omega_{p}= \{T_{\langle (0,q)\rangle},$ $ T_{\langle(p,q)\rangle},$ $ T_{\langle(2p,q)\rangle}, \ldots,$ $ T_{\langle((p-1)p,q)\rangle},T_{\langle(p,0)\rangle}\}$; $\Omega_q=\{ T_{\langle (0,p)\rangle}\}$ and $\Omega_1=\{T_{ \langle(0,0)\rangle}\}$ .
	% We have the partition 
	%\begin{equation*}
	%	V(\mathscr{P}(G))= \Omega_{p^2q} \cup \Omega_{p^2} \cup \Omega_{pq} \cup \Omega_{p} \cup \Omega_{q} \cup \Omega_{1}
	%\end{equation*}
	%As $\langle (i,1)\rangle $ be a cyclic subgroup of $G$ of order $p^2q$ for $1 \leq i \leq p-1$, then suppose $\langle (jp,1)\rangle$ be a subgroup of $\langle (i,1)\rangle $ of order $pq$ for $1 \leq i \leq p-1$, $0 \leq j \leq p-1$. That means, $\langle (sp^2+jp, tpq+1) \rangle $ is the subgroup of $\langle(i,1)\rangle$ for some integers $s$ and $t$. This implies that $sp^2+jp =i(tpq+1)$. Then $p(sp+j-itq)= i$,  which is a contradiction. So  $\langle (jp,1)\rangle$ cannot be  a subgroup of $\langle (i,1)\rangle $ of order $pq$ for $1 \leq i \leq p-1$, $0 \leq j \leq p-1$. And we know that a cyclic group has a unique subgroup of a particular order. That is why,
	Each element of $\Omega_{p^2q}$ is adjacent to 
	 $T_{\langle(p,p)\rangle} \in \Omega_{pq}$ and $T_{\langle(p,0)\rangle} \in \Omega_{p}$ .
	 % Similarly, $\langle(p,0)\rangle$ is the only subgroup of $\langle (i,1)\rangle $ of order $p$ for $1 \leq i \leq p-1$. \\
	Each element of  $\Omega_{p^2q}$ is adjacent to distinct element of  $\Omega_{p^2}$ and each element of  $\Omega_{pq}$ is adjacent to distinct element of  $\Omega_{p}$   because every cyclic group has a unique subgroup of particular order.  It is to be noted that the rows and columns of  $A(\mathscr{P}(G))$  are indexed in order by the elements of $\Omega_{p^2q}$, elements of $\Omega_{p^2}$, elements of $\Omega_{pq}$, elements of $\Omega_{p}$, elements of $\Omega_{q}$, and element of $\Omega_{1}$. So, $A(\mathscr{P}(G))$ is given by

	\setcounter{MaxMatrixCols}{20}
	
	\newcommand{\hugematrix}{$\left[\begin{smallmatrix*}[r]
		I_{p\times p} \bigotimes (J-I)_{\phi(p^2q)\times \phi(p^2q)}& M_{p,p} ^T \bigotimes J_{\phi(p^2q)\times \phi(p^2)}& E_{(p+1)\times p}^T \bigotimes J_{\phi(p^2q)\times \phi(pq)} &  E_{(p+1)\times p}^T \bigotimes J_{\phi(p^2q)\times \phi(p)}& J_{p\times 1} \bigotimes J_{\phi(p^2q)\times \phi(q)}& J_{p\times 1} \bigotimes J_{\phi(p^2q)\times 1}\\
			M_{p\times p} \bigotimes J_{\phi(p^2)\times \phi(p^2q)} & I_{p\times p} \bigotimes (J-I)_{\phi(p^2)\times \phi(p^2)}& O_{p\phi(p^2)\times (p+1)\phi(pq)}& E_{(p+1)\times p}^T \bigotimes J_{\phi(p^2)\times \phi(p)}& O_{p\phi(p^2)\times 1} & J_{p\times 1} \bigotimes J_{\phi(p^2)\times 1}\\
			E_{(p+1)\times p} \bigotimes J_{\phi(pq)\times \phi(p^2q)}& O_{(p+1)\phi(pq)\times p\phi(p^2)}& I_{(p+1)\times (p+1)} \bigotimes (J-I)_{\phi(pq)\times \phi(pq)}& N_{(p+1)\times (p+1)} ^T \bigotimes J_{\phi(pq)\times \phi(p)}& J_{(p+1)\times 1} \bigotimes J_{\phi(pq)\times \phi(q)}& J_{(p+1)\times 1} \bigotimes J_{\phi(pq)\times 1}\\
			E_{(p+1)\times p} \bigotimes J_{\phi(p)\times \phi(p^2q)}& 	E_{(p+1)\times p} \bigotimes J_{\phi(p)\times \phi(p^2)}& N_{(p+1)\times (p+1)} \bigotimes J_{\phi(p)\times \phi(pq)}& I_{(p+1)\times (p+1)} \bigotimes (J-I)_{\phi(p)\times \phi(p)}& O_{p\phi(p)\times \phi(q)}& J_{p\times 1} \bigotimes J_{\phi(p)\times 1}\\
			J_{1\times p} \bigotimes J_{\phi(q)\times \phi(p^2q)}& O_{1\times p phi(p^2)}& J_{1 \times (p+1)} \bigotimes J_{\phi(q)\times \phi(pq)}& O_{1\times (p+1)phi(p)}& (J-I)_{\phi(q)\times \phi(q)}& J_{\phi(q)\times 1}\\
			J_{1 \times p\phi(p^2q)}& J_{1\times p \phi(p^2)}& J_{1\times (p+1)\phi(pq)}& J_{1\times p\phi(p)}& J_{1 \times\phi(q)}& 0			 
		\end{smallmatrix*}\right]$}

	\begin{multline*}
		  \clipbox*{0pt {-1.1\depth} {.72\width} {1.1\height}}{\hugematrix}\\ 
		\clipbox*{{.72\width} {-1.1\depth} {\width} {1.1\height}}{\hugematrix}
	\end{multline*}
	
where $M,N$ are permutations matrices.

%Then the quotient matrix of $A(\mathscr{P}(G))$ is 
%\begin{equation*}
%	Q=\scriptsize
%	\begin{bmatrix}
%		(\phi(p^2q)-1)I_{p \times p}& \phi(p^2)M_{p \times p} ^T& \phi(pq)E_{(p+1)\times p} ^T& \phi(p)E_{(p+1)\times p} ^T& \phi(q)J_{p\times 1}& J_{p\times 1}\\
%		\phi(p^2q)M_{p\times p}& (\phi(p^2)-1)I_{p\times p}& O_{p\times (p+1)}& \phi(p)E_{(p+1)\times p} ^T& O_{p\times 1}&J_{p\times 1}\\
%		\phi(p^2q)E_{(p+1)\times p}& O_{(p+1)\times p}& (\phi(pq)-1)I_{(p+1)\times (p+1)}& \phi(p)N_{(p+1)\times (p+1)} ^T& \phi(q)J_{(p+1)\times 1}& J_{(p+1)\times 1}\\
%		\phi(p^2q)E_{(p+1)\times p}& \phi(p^2)E_{(p+1)\times p}& \phi(pq)N_{(p+1)\times (p+1)}& (\phi(p)-1)I_{(p+1)\times (p+1)}& O_{1\times (p+1)}&J_{(p+1)\times 1}\\
%		\phi(p^2q)J_{1\times p}& O_{1\times p}& \phi(pq)J_{1\times (p+1)}&O_{1 \times (p+1)}&\phi(q)-1&1\\
%		\phi(p^2q)J_{1\times p}& \phi(p^2)J_{1\times p}& \phi(pq)J_{1\times (p+1)}& \phi(p)J_{1\times (p+1)}&1&0
%	\end{bmatrix}
%\end{equation*}

By Lemma \ref{lema2.5}  and Theorem \ref{theorem5.1}, we have 
%\begin{equation*}\label{eq5}
%	\psi(A(\mathscr{P}(G),x)= \psi(Q,x) \frac{[\psi (K_{\phi(p^2q)},x)]^{p }[\psi (K_{\phi(p^2)},x)]^{p } [\psi (K_{\phi(pq)},x)]^{p+1 } [\psi (K_{\phi(p)},x)]^{p+1} [\psi (K_{\phi(q)},x)]}{(x-\phi(p^2q)+1)^{p} (x-\phi(p^2)+1)^{p} (x-\phi(pq)+1)^{p+1} (x-\phi(p)+1)^{p+1} (x-\phi(q)+1)}
%\end{equation*}

\begin{equation}\label{eq5.3}
	\psi(A(\mathscr{P}(G),x)= \psi(Q,x) (1+x)^{[p(\phi(p^2q)+\phi(p^2)-2)+(p+1)(\phi(pq)+\phi(p)-2)+\phi(q)-1]}
\end{equation}
where $\psi(Q,x)$ is
\begin{equation*}
	\scriptsize \begin{vmatrix}
		  	(\phi(p^2q)-1-x)I_{p\times p}& \phi(p^2)M_{p\times p} ^T& \phi(pq)E_{(p+1)\times p} ^T& \phi(p)E_{(p+1)\times p} ^T& \phi(q)J_{p\times 1}& J_{p\times 1}\\
		  \phi(p^2q)M_{p\times p}& (\phi(p^2)-1-x)I_{p\times p}& O_{p\times (p+1)}& \phi(p)E_{(p+1) \times p} ^T& O_{p\times 1}&J_{p \times 1}\\
		  \phi(p^2q)E_{(p+1) \times p}& O_{(p+1)\times p}& (\phi(pq)-1-x)I_{(p+1)\times (p+1)}& \phi(p)N_{(p+1)\times (p+1)} ^T& \phi(q)J_{(p+1) \times 1}& J_{(p+1) \times 1}\\
		  \phi(p^2q)E_{(p+1) \times p}& \phi(p^2)E_{(p+1)\times p}& \phi(pq)N_{(p+1) \times (p+1)}& (\phi(p)-1-x)I_{(p+1)\times (p+1)}& O_{1\times (p+1)}&J_{(p+1)\times 1}\\
		  \phi(p^2q)J_{1\times p}& O_{1\times p}& \phi(pq)J_{1\times (p+1)}&O_{1\times p+1}&\phi(q)-1-x&1\\
		  \phi(p^2q)J_{1\times p}& \phi(p^2)J_{1\times p}& \phi(pq)J_{1\times (p+1)}& \phi(p)J_{1\times (p+1)}&1&-x
	\end{vmatrix}
\end{equation*}
After  suitable  interchanging of rows and columns simultaneously (no effect on sign of determinant), we can convert $M,N$ to  identity matrices and this will not effect other blocks.  Then

\begin{equation*}
\psi(Q,x)=	\tiny \begin{vmatrix}
		(\phi(p^2q)-1-x)I_{p \times p}& \phi(p^2)I_{p\times p}& \phi(pq)E_{(p+1)\times p} ^T& \phi(p)E_{(p+1) \times p} ^T& \phi(q)J_{p\times 1}& J_{p\times 1}\\
		\phi(p^2q)I_{p\times p}& (\phi(p^2)-1-x)I_{p\times p}& O_{p\times (p+1)}& \phi(p)E_{(p+1)\times p} ^T& O_{p\times 1}&J_{p\times 1}\\
		\phi(p^2q)E_{(p+1)\times p}& O_{(p+1)\times p}& (\phi(pq)-1-x)I_{(p+1)\times (p+1)}& \phi(p)I_{(p+1)\times (p+1)}& \phi(q)J_{(p+1)\times 1}& J_{(p+1)\times 1}\\
		\phi(p^2q)E_{(p+1)\times p}& \phi(p^2)E_{(p+1)\times p}& \phi(pq)I_{(p+1)\times (p+1)}& (\phi(p)-1-x)I_{(p+1)\times (p+1)}& O_{1,(p+1)}&J_{(p+1)\times 1}\\
		\phi(p^2q)J_{1\times p}& O_{1\times p}& \phi(pq)J_{1\times (p+1)}&O_{1\times (p+1)}&\phi(q)-1-x&1\\
		\phi(p^2q)J_{1\times p}& \phi(p^2)J_{1\times p}& \phi(pq)J_{1\times (p+1)}& \phi(p)J_{1\times (p+1)}&1&-x
	\end{vmatrix}
\end{equation*}

Taking common $\phi(p^2q), \phi(p^2), \phi(pq),\phi(p)$ and $\phi(q)$ from the first $p$ columns, the next $p$ columns, the next $(p+1)$ columns, the next $(p+1)$ columns and  the last second column respectively.  Let $ a= \frac{\phi(p^2q)-1-x}{\phi(p^2q)}$, $b =\frac{\phi(p^2)-1-x}{\phi(p^2)} $, $c=\frac{\phi(pq)-1-x}{\phi(pq)}$, $d= \frac{\phi(p)-1-x}{\phi(p)}$, $e=\frac{\phi(q)-1-x}{\phi(q)}$ and $f=-x$.
Then, we get
\begin{equation*}
\psi(Q,x)=	\scriptsize \phi(p^2q)^{p}\phi(p^2)^{p}\phi(pq)^{p+1} \phi(p)^{p+1}\phi(q) \begin{vmatrix}
		aI_{p\times p}& I_{p\times p}&E_{(p+1) \times p} ^T& E_{(p+1)\times p} ^T& J_{p\times 1}& J_{p\times 1}\\
		I_{p\times p}&bI_{p\times p}& O_{p\times (p+1)}& E_{(p+1)\times p} ^T& O_{p\times 1}&J_{p\times 1}\\
		E_{(p+1)\times p}& O_{(p+1)\times p}& cI_{(p+1)\times p+1}& I_{(p+1)\times p+1}& J_{(p+1)\times 1}& J_{(p+1)\times 1}\\
		E_{(p+1)\times p}& E_{(p+1)\times p}& I_{(p+1)\times (p+1)}& dI_{(p+1)\times (p+1)}& O_{(p+1)\times 1}&J_{(p+1)\times 1}\\
		J_{1\times p}& O_{1\times p}& J_{1\times (p+1)}&O_{1\times (p+1)}&e&1\\
		J_{1\times p}& J_{1\times p}& J_{1\times (p+1)}& J_{1\times (p+1)}&1&f
	\end{vmatrix}
\end{equation*}
Suppose $w= \phi(p^2q)^{p}\phi(p^2)^{p}\phi(pq)^{p+1} \phi(p)^{p+1}\phi(q)$.
Apply the row operations $R_{p+i} \to R_{p+i}-bR_{i}$ for $1 \leq i \leq p$, $R_{4p+2} \to R_{4p+2} -(R_1+R_2+ \ldots+R_{p})$ and $R_{4p+4} \to R_{4p+4} -(R_1+R_2+ \ldots R_p)$. We get
\begin{equation*}
\psi(Q,x)=	w \scriptsize \begin{vmatrix}
 	aI_{p \times p}& I_{p\times p}&E_{(p+1)\times p} ^T& E_{(p+1)\times p} ^T& J_{p\times 1}& J_{p\times 1}\\
 	(1-ab)I_{p\times p}&O_{p\times p}& -bE_{(p+1)\times p} ^T& (1-b)E_{(p+1)\times p} ^T& -bJ_{p\times 1}&(1-b)J_{p\times1}\\
 	E_{(p+1)\times p}& O_{(p+1)\times p}& cI_{(p+1)\times (p+1)}& I_{(p+1)\times (p+1)}& J_{(p+1)\times 1}& J_{(p+1)\times 1}\\
 	(1-a)E_{(p+1)\times p}& O_{(p+1)\times p}& (I-pL)_{(p+1)\times (p+1)}& (dI-pL)_{(p+1)\times (p+1)}& -pL_{(p+1)\times 1}&(J-pL)_{(p+1)\times 1}\\
 	J_{1\times p}& O_{1\times p}& J_{1\times(p+1)}&O_{1\times (p+1)}&e&1\\
 	(1-a)J_{1\times p}& O_{1\times p}& (J-pL)_{1\times (p+1)}& (J-pL)_{1\times (p+1)}&1-p&f-p
 \end{vmatrix}
\end{equation*}
Expanding the determinant along the  columns $C_{p+1}, C_{p+2},\ldots, C_{2p}$, we get
\begin{equation*}
	\psi(Q,x)=	w \scriptsize \begin{vmatrix}		
		(1-ab)I_{p\times p}& -bE_{(p+1) \times p} ^T& (1-b)E_{(p+1) \times p} ^T& -bJ_{p\times 1}&(1-b)J_{p\times 1}\\
		E_{(p+1)\times p}& cI_{(p+1)\times (p+1)}& I_{(p+1)\times (p+1)}& J_{(p+1)\times 1}& J_{(p+1)\times 1}\\
		(1-a)E_{(p+1) \times p}& (I-pL)_{(p+1)\times (p+1)}& (dI-pL)_{(p+1)\times (p+1)}& -pL_{(p+1)\times 1}&(J-pL)_{(p+1)\times 1}\\
		J_{1\times p}& J_{1\times (p+1)}&O_{1\times(p+1)}&e&1\\
		(1-a)J_{1\times p}& (J-pL)_{1\times (p+1)}& (J-pL)_{1\times (p+1)}&1-p&f-p
	\end{vmatrix}
\end{equation*}
Applying the column operations $C_{p+i} \to C_{p+i} -cC_{2p+1+i}$ for $1 \leq i \leq p$, $C_{3p+3} \to C_{3p+3}- (C_{2p+2}+ C_{2p+3}+\ldots+ C_{3p+1})$ and $C_{3p+4} \to C_{3p+4}- (C_{2p+2}+ C_{2p+3}+\ldots+ C_{3p+1})$, the determinant is

\begin{equation*}
		w \scriptsize \begin{vmatrix}		
		(1-ab)I_{p \times p}& -bE_{(p+1)\times p} ^T& (1-b)E_{(p+1)\times p} ^T& -bJ_{p\times 1}&(1-b)J_{p\times 1}\\
		E_{(p+1)\times p}& cL_{(p+1)\times (p+1)}& I_{(p+1)\times (p+1)}& L_{(p+1)\times 1}& L_{(p+1)\times 1}\\
		(1-a)E_{(p+1)\times p}& ((1-cd)I-(p-cd)L)_{(p+1)\times (p+1)}& (dI-pL)_{(p+1)\times (p+1)}& (-dJ+(d-p)L)_{(p+1)\times 1}&((1-d)J+(d-p)L)_{(p+1)\times 1}\\
		J_{1\times p}& J_{1\times (p+1)}&O_{1\times (p+1)}&e&1\\
		(1-a)J_{1\times p}& ((1-c)J-(p-c)L)_{1\times (p+1)}& (J-pL)_{1\times (p+1)}&1-2p&f-2p
	\end{vmatrix}
\end{equation*}
Expanding along the rows $R_{p+1},R_{p+2},\ldots,R_{2p}$,  $\psi(Q,x)$ is 
\begin{equation*}
		w \scriptsize \begin{vmatrix}		
		(1-ab)I_{p \times p}& -bE_{(p+1) \times p} ^T& (1-b)J_{p\times 1} ^T& -bJ_{p\times 1}&(1-b)J_{p\times 1}\\
		J_{1\times p}& cL_{1\times (p+1)}& 1& 1&1\\
		(1-a)E_{(p+1)\times p}& ((1-cd)I-(p-cd)L)_{(p+1)\times (p+1)}& (d-p)L_{(p+1)\times 1}& (-dJ+(d-p)L)_{(p+1)\times 1}&((1-d)J+(d-p)L)_{(p+1)\times 1}\\
		J_{1\times p}& J_{1\times (p+1)}&0&e&1\\
		(1-a)J_{1\times p}& ((1-c)J-(p-c)L)_{1\times (p+1)}& (1-p)&1-2p&f-2p
	\end{vmatrix}
\end{equation*}
Apply the column operations $C_i \to C_i -C_p$ for $1 \leq i \leq p-1$ and $C_j \to C_j- C_{2p}$ for $p+1 \leq j \leq 2p-1$, we get
\begin{equation*}
	\psi(Q,x)=	w \scriptsize \begin{vmatrix}		
		(1-ab)(I-E+L)_{p\times p}&O_{p\times p}& -bJ_{p\times 1}& (1-b)J_{p\times 1}& -bJ_{p\times 1}&(1-b)J_{p\times 1}\\
		L_{1\times p}& O_{1\times p}&c&1& 1&1\\
		O_{p\times p}&(1-cd)(I-E+L)_{p\times p}&O_{p\times 1}&-dJ_{p\times 1}& (1-d)J_{p\times 1}\\
		(1-a)L_{1\times p}& O_{1\times p}&(1-p)&d-p&-p&1-p\\
		L_{1\times p}& L_{1\times p}&1&0&e&1\\
		(1-a)L_{1\times p}& L_{1\times p}& 1-p&1-p&1-2p&f-2p
	\end{vmatrix}
\end{equation*}
Taking common $(1-ab)$ from the first $(p-1)$ columns and $(1-cd)$ from $C_{p+1},C_{p+2},\ldots, C_{2p-1}$, then  $\psi(Q,x)$ is
\begin{equation*}
	w(1-ab)^{p-1}(1-cd)^{p-1}\scriptsize
	\begin{vmatrix}
		(I-E+(1-ab)L)_{p\times p}&O_{p\times p}&-bJ_{p\times 1}&(1-b)J_{p\times 1}&-bJ_{p\times 1}&(1-b)J_{p\times 1}\\
		L_{1\times p}& O_{1\times p}&c&1&1&1\\
		O_{p\times p}&(I-E+(1-cd)L)_{p\times p}&O_{p\times 1}&O_{p\times 1}&-dJ_{p\times 1}&(1-d)J_{p\times 1}\\
		(1-a)L_{1\times p}& O_{1\times p}&1-p&d-p&-p&1-p\\
		L_{1\times p}&L_{1\times p}&1& 0& e&1\\
		(1-a)L_{1\times p}&(1-c)L_{1\times p}& 1-p& 1-p& 1-2p&f-2p
	\end{vmatrix}
\end{equation*}
Apply the column operations  $C_p \to C_p -(C_1+C_2+\ldots+C_{p-1})$, $C_{2p+1} \to C_{2p+1} -(C_{p+2}+C_{p+3}+\ldots+C_{2p})$,  $\psi(Q,x)$ is
\begin{equation*} 
	 w(1-ab)^{p-1}(1-cd)^{p-1}\scriptsize
	\begin{vmatrix}
		(I-abL)_{p \times p}&O_{p \times p}&-bJ_{p \times 1}&(1-b)J_{p\times 1}&-bJ_{p \times1}&(1-b)J_{p \times1}\\
		L_{1 \times p}& O_{1 \times p}&c&1&1&1\\
		O_{p \times p}&(I-cdL)_{p \times p}&O_{p \times 1}&O_{p\times 1}&-dJ_{p\times 1}&(1-d)J_{p\times 1}\\
		(1-a)L_{1\times p}& O_{1\times p}&1-p&d-p&-p&1-p\\
		L_{1\times p}&L_{1\times p}&1& 0& e&1\\
		(1-a)L_{1\times p}&(1-c)L_{1\times p}& 1-p& 1-p& 1-2p&f-2p
	\end{vmatrix}
\end{equation*}
Expand along the columns $C_1,C_2,\ldots,C_{p-1}$ and $C_{p+1},C_{p+2},\ldots,C_{2p-1}$, we get
\begin{equation*}
	\psi(Q,x)=  w(1-ab)^{p-1}(1-cd)^{p-1} 
\begin{vmatrix}
1-ab& 0&-b&1-b&-b&1-b\\
1&0&c&1&1&1\\
0&1-cd&0&0&-d&1-d\\
1-a&0&1-p&d-p&-p&1-p\\
1&1&1&0&e&1\\
1-a&1-c&1-p&1-p&1-2p&f-2p
\end{vmatrix}
\end{equation*}
The factor $(1-ab)^{p-1}= \big(\frac{(1+x)(p^2q-pq-1-x)}{\phi(p^2q)\phi(p^2)}\big)^{p-1}$ gives the  eigenvalues $ -1$ and $(p^2q-pq-1-x)$ with both multiplicities $p-1$, and  $(1-cd)^{p-1}= \big(\frac{(1+x)(pq-q-1-x)}{\phi(p^2)\phi(p)}\big)^{p-1}$ gives the  eigenvalues $ -1$ and $(pq-q-1-x)$ with both multiplicities $p-1$.
By putting value of $w$,  $\psi(Q,x)$ is
\begin{equation*}
\phi(p^2q)\phi(p^2)\phi(q){\phi(pq)}^2 {\phi(p)}^2 (1+x)^{2p-2}(p^2q-pq-1-x)^{p-1} (pq-q-1-x)^{p-1}det(C) 
\end{equation*}
where  $det(C)$ as given in Equation (\ref{eq6}). Finally by Equation (\ref{eq5.3}), we get multiplicity of $-1$ is $p^3q-2p-6$.
\end{proof}

\end{document}